 \documentclass[a4paper,10pt]{article}

\usepackage{amsfonts,amssymb,amsmath}
\usepackage{graphicx,graphics,epsfig,color}
\usepackage{rgsMacros}
\usepackage{subcaption}
\usepackage{hyperref}

\usepackage{enumitem}
\usepackage{balance}

\definecolor{olivegreen}{rgb}{0.14,0.29,0}

\newtheorem{exe}{Example}
\newtheorem{corol}{Corollary}
\newtheorem{ass}{Assumption}
\newtheorem{proper}{Property}
\newtheorem{defin}{Definition}
\newtheorem{prob}{Problem}
\newtheorem{cla}{Claim}
\newtheorem{rem}{Remark}
\newtheorem{lem}{Lemma}
\newtheorem{prop}{Proposition}
\newtheorem{thm}{Theorem}
\newtheorem{fct}{Fact}
\newenvironment{lemma}{\begin{lem}}{\hfill $\square$ \end{lem}}

\newenvironment{remark}{\begin{rem} \rm}{\hfill $\bullet$ \end{rem}}
\newenvironment{assumption}{\begin{ass}}{\hfill $\bullet$ \end{ass}}
\newenvironment{property}{\begin{proper}}{\hfill $\bullet$ \end{proper}}
\newenvironment{theorem}{\begin{thm}}{\hfill $\square$ \end{thm}}
\newenvironment{definition}{\begin{defin}}{\hfill $\bullet$ \end{defin}}
\newenvironment{problem}{\begin{prob}}{\hfill $\bullet$ \end{prob}}

 \newenvironment{proof}{\noindent {\it Proof.}}{\hfill \mbox{\footnotesize $\blacksquare$}}

\newif\ifitsdraft
\def\itsdraft{\global\itsdrafttrue}

 \itsdraft  

\title{\LARGE \bf Boundary Control of the Kuramoto-Sivashinsky Equation Under Intermittent Data Availability: With Proofs}

\author{M. Maghenem,  C. Prieur,  and E. Witrant
\thanks{ M. Maghenem,  C. Prieur,  and E. Witrant are with Universit\'e Grenoble Alpes, CNRS, Grenoble-INP, GIPSA-lab, F-38000, Grenoble, France (e-mail: mohamed.maghenem,christophe.prieur,emmanuel.witrant@gipsa- lab.fr).  
This work has been partially supported by MIAI@Grenoble Alpes (ANR-19-P3IA-0003), the French program Investissement d’Avenir,  and by ANR grant HANDY (ANR-18-CE40-0010).  }}

\begin{document}

\maketitle

\begin{abstract}
In this paper,  two boundary controllers are proposed to stabilize the origin of the nonlinear Kuramoto-Sivashinsky equation under intermittent measurements.   More precisely,  the spatial domain is divided into two sub-domains.  The state of the system on the first sub-domain is measured along a given interval of time,  and the state on the remaining sub-domain is measured along another interval of time.  Under the proposed sensing scenario,  we control the considered equation by designing the value of the state at three isolated spatial points,  the two extremities of the spatial domain plus one inside point.  Furthermore,  we impose a null value for the spatial gradient of the state at these three locations.  Under such a control loop,  we propose two types of controllers and we analyze the stability of the resulting closed-loop system in each case.  The paper is concluded with some discussions and future works. 
\end{abstract}

\section{Introduction}

Partial differential equations (PDE)s have numerous applications in many engineering fields including fluid flows in conservation laws  \cite{kuramoto1980instability},  flexible structures  \cite{hac1993sensor},  electromagnetic waves,  and quantum mechanics \cite{rouchon2008quantum}.  The control design for PDEs is a key step to guarantee that the related process achieves a desired behavior in closed loop, i.e.,  a state of interest converges (in an appropriate norm) to an invariant set \cite{schneider2017nonlinear,karafyllis2019input},  or tracks the state of a driving process \cite{kocarev1997synchronizing,guo2020robust}.   Before designing the control input, it is important to know the control actions allowed by the physical process.  Indeed,  some processes allows to act on the dynamics at every spatial point and for all time \cite{armaou1999nonlinear}.  However, in some other processes,  we act intermittently in time or in space \cite{kang2018distributed, tasev2000synchronization}.  Furthermore,  in some scenarios, we can reset the state intermittently in time and at every spatial point \cite{khadra2005impulsive}, but in other scenarios,  we reset the state only at some spatial points \cite{kocarev1997synchronizing},  the latter case corresponds  to the well-studied boundary-control paradigm \cite{krstic2008boundary}.  On the other hand,  it is important to know the outputs available for input design.  In some cases,  we measure the state at every spatial point all the time \cite{krstic2008adaptive}. However,  in most realistic scenarios,  we measure only intermittently in space and time \cite{kang2018distributed,1102875}.   The feedback law,  in consequence,  must  adapt to each of these control and sensing scenarios.  
\ifitsdraft
Note that in the context of ordinary differential equations, temporally-intermittent control strategies (by acting on the dynamics or by doing a state reset) are well studied; see \cite{laila20063,yang2001impulsive}.  Furthermore,  the spatially-intermittent control of finite-dimensional systems corresponds to the scenario where the input affects directly some states but not all of them.  
\fi

Some intermittent control strategies for PDEs are available  in the literature.  In \cite{khadra2005impulsive}, the Gray-Scott and the Kuramoto-Sivashinsky equations are controlled via a periodic reset of the state using impulsive systems theory.   Furthermore,  in \cite{kang2018distributed}, the Kuramoto-Sivashinsky equation is controlled via a periodic update of the input affecting the right-hand side using sample-data control techniques.  A common feature among the aforementioned works is  that the PDEs are controlled at every spatial point.  When controlling  PDEs at isolated spatial points or intervals, the existing control literature considers only the particular case of boundary control.   However,  the physics community,  since the late nineties,  has shown an intensive interest in general  spatially- and temporally-intermittent control of PDEs.  For example,  in \cite{kocarev1997synchronizing},  the Gray-Scott equation is controlled by resetting the state periodically in time and at periodically separated spatial points.  In \cite{tasev2000synchronization, parmananda1997generalized}, the Kuramoto-Sivashinsky equation is controlled by acting on the right-hand side at periodically separated spatial intervals.  In  \cite{junge2000synchronization},  the Ginzburg-Landau equation is studied following the strategy in \cite{tasev2000synchronization}.   The results in  the aforementioned physics literature are guaranteed via simulations, and experiments in some cases.  To the best of our knowledge,  a rigorous study of the aforementioned problems is not available in the literature.

The Kuramoto-Sivashinsky equation is one of the well-studied PDEs in control literature.  In particular,  different types of boundary controllers are proposed to stabilize the origin in a given norm.  For example, in \cite{katz2020finite},  the linear Kuramoto-Shavinsky equation is transformed into an equivalent finite-dimensional linear system using the Sturm-Liouville decomposition.  As a consequence,  a linear feedback law is assigned to one extremity of the spatial domain.  Furthermore,  the nonlinear equation is studied in \cite{Toshidaref,  liu2001stability,  sakthivel2007non, guzman2019stabilization} under various boundary conditions.  It is important to note  that although the aforementioned results  assume point-wise measurements,  the equation is assumed to be either linear or intrinsically stable.  To the best of our knowledge,  boundary control of the nonlinear Kuramoto-Sivashinski equation without restricting the destabilizing coefficient is  studied only in \cite{coron2015fredholm},  where the state,  on the whole spacial domain,  is assumed to be available all the time.

In this paper,  we study boundary control of the nonlinear Kuramoto-Sivashinsky equation,  without  restricting the destabilizing coefficient,  under intermittent measurements.   Strictly speaking,  we  measure the state on a given spatial sub-domain along a given interval of time and, then, we measure the state on the remaining spatial sub-domain along another time interval.   As a result,  we do not measure the state,  on the whole spatial domain,  at the same time.   Under the proposed sensing scenario,   we control the considered equation by designing  the state at three isolated points,  the two extremities of the spatial domain plus one inside point.  Furthermore, we impose a null value for the spatial gradient of the state at these three locations.   Two types of design approaches are proposed.  In the first one,  we design feedback laws at the two extremities and set the input  to zero at the inside point.  In the second case,  we design a feedback law at one extremity and at the inside point,  and set the input to zero at the remaining extremity.   The stability properties of the resulting closed-loop system are analyzed in each case using Lyapunov methods.

The remainder of the paper is organized as follows.  The problem formulation is in Section \ref{SecProbForm}.  The proposed Lyapunov-based approach is described in Section \ref{secGenApp}.  The main results are in Sections \ref{Sec.DBC1} and \ref{Sec.DBC2}, respectively.   Finally,  the paper is concluded by some discussions and future works.

\textbf{Notation.} 
For $x \in \mathbb{R}^n$, $(a,b) \in \mathbb{R} \times \mathbb{R}$ with $a<b$, and a function $z : [a,b] \rightarrow \mathbb{R}$,  we let $|x| := \sqrt{x^\top x}$, $|| z || := \sqrt{\int^b_a |z(x)|^2 dx}$, $z_x := \frac{\partial z}{\partial x}$, $z_{xx} := \frac{\partial^2 z}{\partial x^2}$, $z_{xxx} := \frac{\partial^3 z}{\partial x^3}$, and $z_{xxxx} := \frac{\partial^4 z}{\partial x^4}$. Furthermore, we say that $z \in \mathcal{L}^2(a,b)$ if $|| z ||$ is finite, $z \in H^1(a,b)$ if $ ||z|| + || z_x ||$ is finite and $z \in H^2(a,b)$ if $|| z || + || z_{x} || + || z_{xx} ||$ is finite.  Moreover, we say that $z \in H^1_{o}(a,b)$ if $z \in H^1(a,b)$ and $z(a) = z (b) = 0$, and we say that $z \in H^2_{o}(a,b)$ if $z \in H^2(a,b)$ and $z(a) = z(b) = z_{x}(a) = z_{x} (b) = 0$.  For a matrix $M \in \mathbb{R}^{n \times n}$, $\det M$  denotes the determinant of the matrix $M$.  To avoid heavy notations, for a function $u : [a,b] \times \mathbb{R}_{\geq 0} \rightarrow \mathbb{R}$,  the time dependence is implicit when we write $u(x)$.  In other words,  we use $u(x)$, to express $u(x,t)$.  Moreover, we use $\dom u$ to denote the domain of definition of the map $t \mapsto u(x,t)$ for all $x\in [a,b]$.  Finally,  $\kappa : \mathbb{R}_{\geq 0} \rightarrow \mathbb{R}_{\geq 0}$ is a class $\mathcal{K}$ function if it is continuous,  increasing,  and $\kappa(0) = 0$.
  
\section{Problem Formulation} \label{SecProbForm}
Consider the nonlinear Kuramoto-Sivashinsky equation given by 
\begin{align*}
\Sigma_1 : u_t  =  -u u_x - \lambda_1 u_{xx} - u_{xxxx} \qquad x \in [0,L],
\end{align*}
where $L>0$ and $\lambda_1 \geq 0$ are constant coefficients assumed to be known.    The boundary conditions of $\Sigma_1$ will be eventually specified.  

\ifitsdraft
Note that some of the boundary variables are control inputs to be designed   to stabilize the trivial solution to $\Sigma_1$ in the 
$\mathcal{L}^2$ sense. 
\fi
 
\subsection{Intermittent Sensing}  \label{Sec.sens}

For some $Y \in (0,L)$, we measure the state $u$ on the spatial domain $[0,Y]$ for some interval of time and, then, we measure $u$ on the domain $[Y,L]$ along another time interval.   In particular,   we do not measure $u$ along the whole line $[0,L]$ simultaneously.   The  proposed sensing scenario has the following  motivations.
 
 In network control systems \cite{zhang2001stability}.  Assume that we dispose of two sensors, the first one measures $u([0,Y])$ and the second one measures  $u([Y,L])$.   The two sensors share the same channel when sending their data to the controller.  Hence, $u([0,Y])$ is available to the controller only for some interval of time and $u([Y,L])$ only for another interval of time.  The same reasoning can be extended if a network of sensors share the same communication channel with the controller \cite{sun2009networked}.  Also,  when using mobile or scanning sensors; see  \cite{zhang2019fuzzy} and  \cite{khapalov1992observability}, respectively.  Assume that we dispose of one mobile sensor measuring the state $u([0,Y])$ for some interval of time, then, it changes location to measure $u([Y,L])$ for another interval of time. Note that, here, we are neglecting the dynamics of the mobile sensor compared to the dynamics of $\Sigma_1$.  
 This reasoning can be extended to a network of mobile sensors as in \cite{mu2014improving,5406054,ZHANG20201299}.

Strictly speaking,  we assume the existence of a  sequence of times $\{t_i\}^{\infty}_{i=1}$, with $t_1 = 0$ and $t_{i+1} > t_i$,  such that 
\begin{itemize}
\item $u([0,Y], t)$ is available  for all $t \in \bigcup^{\infty}_{k = 1} [t_{2k-1},t_{2k})$.
\item $u([Y,L], t)$ is available for all 
$t \in \bigcup^{\infty}_{k = 1} [t_{2k},t_{2k+1})$.
\end{itemize}
Next,  we consider the following assumption:
\begin{assumption} \label{ass1}
There exist $\bar{T}_1$,  $\bar{T}_2>0$ such that, for each $k \in \{1,2,... \}$,  $ t_{2k} - t_{2k-1} = \bar{T}_1 $ and $t_{2k+1} - t_{2k} = \bar{T}_2$. 
\end{assumption}

\begin{remark} \label{rem3}
The proposed sensing scenario can be generalized if we decompose the interval $[0,L]$ into many  sub-intervals instead of only two.  In this case, we define an increasing sequence $\{Y_i\}^{N}_{i=1}$, with $Y_1 = 0$ and $Y_N = L$ such that, for each $i \in \{1,2,...,N\}$, $u([Y_i,Y_{i+1}], t)$ is available to the controller for all 
$t \in \bigcup^{\infty}_{k = 1} \left[ t_{(N-1)(k-1) + i},t_{(N-1)(k-1) + i + 1} \right)$.
\end{remark}

\begin{remark} \label{rem4}
Note that most of existing results on boundary stabilization of 
$\Sigma_1$ assume point-wise measurements; namely,  the state $u$  is measured at the isolated points where the control is effective;  see for example \cite{Toshidaref,  liu2001stability,  sakthivel2007non, guzman2019stabilization,  katz2020finite}.  However, in the aforementioned results either $\lambda_1$ is assumed to be sufficiently small or the linear version of $\Sigma_1$ is considered.  To the best of our knowledge,  boundary 
control of the nonlinear Kuramoto-Sivashinski equation without restricting the size of $\lambda_1$ is 
studied only in \cite{coron2015fredholm},  where the availability of $u$  on the whole line $[0,L]$ and for all time is assumed.  
\end{remark}

\subsection{Boundary Control}

We propose to control $\Sigma_1$ at three different locations; 
namely,  at $x = 0$,  $x = Y$, and  $x = L$.   
As a result,  we view $\Sigma_1$ as a system of two PDEs interconnected by a boundary constraint.  That is,   we consider the system 
\begin{align*} 
\Sigma_2:
\left\{
\begin{matrix}
w_t =  -w w_x - \lambda_1 w_{xx} - w_{xxxx} \qquad x \in [0,Y] 
\\
v_t =   -v v_x - \lambda_1 v_{xx} - v_{xxxx} \qquad x \in [Y,L]
\end{matrix}
\right.
\end{align*}
under the boundary constraints
\begin{equation}
\label{eqBC1}
\begin{aligned} 
 w(Y) = v(Y) ~ \text{and} ~  w_x(Y) = v_x(Y).
\end{aligned}
\end{equation}
Condition \eqref{eqBC1} guarantees that $w$, defined on $[0,Y]$, is a continuously differentiable extension of $v$, defined on $[Y,L]$, and vice-versa.  

Furthermore,  each equation in $\Sigma_2$ allows for four boundary conditions, counting the two conditions in \eqref{eqBC1}, we conclude that we are allowed to control $\Sigma_2$ by imposing six other boundary conditions,  which are given by 
\begin{equation}
\label{eqDerich}  
\begin{aligned} 
v_x(L)  = 0, ~  w_x(Y) = 0,~w_x(0) = 0,  ~
w(0)  = u_1,~w(Y) = u_2,~v(L) = u_3, 
\end{aligned}
\end{equation}
where $(u_1,u_2,u_3)$ are control inputs to be designed.  

Next, we specify the  set of solutions to $\Sigma_2$.

\begin{definition} \label{Def1-}  
A pair $\left( (w,v),  \{u_i\}^3_{i=1} \right)$, with $w : [0,Y] \times \dom w  \rightarrow \mathbb{R}$, $v : [Y,L] \times \dom v \rightarrow \mathbb{R}$, and $u_i : \dom u_i \rightarrow \mathbb{R}$, is a solution pair to  $\Sigma_2$ if $\dom w = \dom v = \dom u_1 = \dom u_2 = \dom u_3$,  $ (w(\cdot,t),v(\cdot,t)) \in H^4(0,Y) \times H^4(Y,L)$,  $(w(x,\cdot),v(x,\cdot))$ is locally absolutely continuous,   \eqref{eqBC1}-\eqref{eqDerich} hold, and   
\begin{align*}
\frac{\partial w}{\partial t} (x,t) & =  -w(x,t) w_x(x,t) - \lambda_1 w_{xx}(x,t) - w_{xxxx}(x,t)  
\\ & \qquad  \text{for almost all}~ (x,t) \in [0,Y] \times \dom w,
\\
\frac{\partial v}{\partial t}(x,t) & =   -v(x,t) v_x(x,t) - \lambda_1 v_{xx}(x,t) - v_{xxxx}(x,t) 
\\ & \qquad   \text{for almost all}~ (x,t) \in [Y,L] \times  \dom v.
\end{align*}
\end{definition}

As a consequence,  we study $\Sigma_1$ under
\begin{equation} \label{eqDerich1}
\begin{aligned} 
(u_x(L),u_x(Y), u_x(0) )  = (0,0,0), \quad 
 (u(0),u(Y),u(L))   =   (u_1,u_2,u_3)
\end{aligned}
\end{equation}
by studying $\Sigma_2$ under \eqref{eqBC1}-\eqref{eqDerich}.  
Hence,  we specify the set of solutions to $\Sigma_1$ as follows:

\begin{definition}  \label{Def1}  
A pair $\left(u,  \{u_i\}^3_{i=1} \right)$,  with $u : [0,L] \times \dom u \rightarrow \mathbb{R}$ and $u_i : \dom u_i \rightarrow \mathbb{R}$,
 is a solution pair to  $\Sigma_1$ if there exists $(w,v)$ such that $\left( (w,v),  \{u_i\}^3_{i=1} \right)$ is a solution pair to $\Sigma_2$,  $\dom u = \dom w = \dom v$,   $u(x,t)  = w(x,t)$ for all $(x,t) \in [0,Y] \times \dom u$,  and $u(x,t) = v(x,t)$  for all $(x,t) \in [Y,L] \times \dom u$.  
\end{definition}

\begin{remark}
In \cite{Toshidaref}, $\Sigma_1$ is considered under the boundary conditions   $u_x(L) =  u_x(0) = 0$,  $u_{xxx}(0) = u_1$,  and $u_{xxx}(L) = u_2$.  Furthermore,  in \cite{coron2015fredholm}, the boundary conditions $u(L) =  u(0) = 0$,  $u_{xx}(0) = 0$,  $u_{xx}(L) = u_2$ are considered.   Compared to the latter two references,  in our work, we use an additional control action at $x = Y$ to handle the intermittent  availability  of the measurements. 
\end{remark}
In the sequel,  we  design $(u_1,u_2,u_3)$ for 
$\Sigma_1$,  by following two  scenarios. First,  we set $u_2 = 0$ and we design $u_1$ and $u_3$.  Second,  we set $u_3 = 0$ and we design $u_1$ and $u_2$.

\ifitsdraft
Recall that,  due to the considered sensing and control scenarios,  we design the inputs using $u([0,Y]) = w([0,Y])$ only, when $ t \in I_1 := \bigcup^{\infty}_{k = 1} [t_{2k-1},t_{2k})$. 
Similarly, we design the inputs using $u([Y,L]) = v([Y,L])$ only when $t \in I_2 := \bigcup^{\infty}_{k = 1} [t_{2k},t_{2k+1})$.
\fi

\section{General Approach} \label{secGenApp}
 
Before designing the control inputs,  we consider the following two Lyapunov function candidates:
\begin{equation} \label{eq.4}
\begin{aligned} 
V_1(w)  :=  \frac{1}{2} \int^{Y}_{0} w(x)^2 dx, 
\quad 
V_2(v)  :=  \frac{1}{2}\int^{L}_{Y} v(x)^2 dx.
\end{aligned}    
\end{equation}

\begin{lemma} \label{lemintegpart}
Along the solutions to $\Sigma_2$ and under \eqref{eqBC1} and \eqref{eqDerich}, we have 
\begin{equation}
\label{eqLyapgen}
\begin{aligned}
\dot{V}_1 = & - \int^{Y}_{0} w_{xx}(x)^2 dx  + 
 \lambda_1 \int^{Y}_{0} w_x(x)^2 dx  -  \frac{u_2^3 - 
u_1^3}{3} - u_2 w_{xxx}(Y)  + u_1 w_{xxx}(0), 
\\
\dot{V}_2 = &  - \int^{L}_{Y} v_{xx}(x)^2 dx  
+  \lambda_1 \int^{L}_{Y} v_x(x)^2 dx  -  \frac{u_3^3 - u_2^3}{3} - u_3 v_{xxx}(L)  + u_2 v_{xxx}(Y).  
\end{aligned}
\end{equation}
\end{lemma}
Next,  we introduce a key result that allows us to upper bound   the terms $\left[- \int^{Y}_{0} w_{xx}(x)^2 dx  + 
 \lambda_1 \int^{Y}_{0} w_x(x)^2 dx \right]$ and $\left[- \int^{L}_{Y} v_{xx}(x)^2 dx  +  \lambda_1 \int^{L}_{Y} v_x(x)^2 dx \right]$ in \eqref{eqLyapgen} using $V_1$, $V_2$, and the inputs $(u_1,u_2,u_3)$.  To this end,  we introduce the following eigenvalue problem.
\begin{problem} \label{prob1}
Given $\lambda \geq 0$,  find the smallest $\delta \in \mathbb{R}$, denoted $\delta_o$, such that
\begin{equation}
\label{eigenprob}
\begin{aligned}
z_{xxxx} & + \lambda z_{xx} = \delta z \qquad x \in [a,b]
\end{aligned}
\end{equation}
admits a nontrivial solution 
$z : [a,b] \rightarrow \mathbb{R}$ in $H^2(a,b)$ satisfying 
\begin{align} \label{eqBC}
z(a) = z(b) = z_x(a) = z_x(b) = 0.
\end{align}
\end{problem}
The following result can be found  in \cite[Lemma 2.1, Theorem 2.1, and Remark 1]{liu2001stability}.
\begin{lemma} \label{lem3}
Given $\lambda \geq 0$, we  let $\delta_o \in \mathbb{R}$ be the corresponding solution to Problem  \ref{prob1}. 
Then,  if $\lambda < 4 \pi^2$ then $\delta_o > 0$.  Furthermore,  if $\lambda > 4 \pi^2$ then $\delta_o < 0$.  Finally,  if $\lambda = 4 \pi^2$ then $\delta_o = 0$.
\end{lemma}
The following lemma can be found in  \cite[Lemma 3.1]{liu2001stability}.
\begin{lemma} \label{lem4-}
Given $\lambda \geq 0$ and $z \in H^2_{o}(a,b)$.   Let $\delta_o \in \mathbb{R}$ be the solution to Problem  \ref{prob1}.  
Then,  
\begin{align} \label{eqlem1-}
\hspace{-0.2cm} - \int^{b}_{a} z_{xx}(x)^2 dx   + \lambda  \int^{b}_{a} z_x(x)^2 dx  \leq  
-\delta_o \int^{b}_{a} z(x)^2 dx.
\end{align}
\end{lemma}

Now, we propose a generalization of Lemma \ref{lem4-}.

\begin{lemma} \label{lem4}
Given $\lambda_1 > 0$  and $z \in H^2(a,b)$ with $z_x(a) = z_x(b) = 0$.    Let $\delta_o$ be the solution to Problem \ref{prob1} with $\lambda := 3 \lambda_1$.  
Then,   for each $\delta \leq \delta_o$,  we have 
\begin{equation} 
\label{eqlem1}
\begin{aligned} 
& - \int^{b}_{a} z_{xx}(x)^2 dx   + \lambda_1 \int^{b}_{a} z_x(x)^2 dx  
\leq   \delta_1 \int^{b}_{a} z(x)^2 dx 
\\ &
+  C_{z1} (z(a),z(b)) + \delta_2  C_{z2}(z(a),z(b)) + \lambda_1 C_{z3}(z(a),z(b)),
\end{aligned}
\end{equation}
where $\delta_1 := \left(|\delta| - 2 \delta \right)/ 3$,  $\delta_2 := \left(4 |\delta| - 2 \delta \right) /3$,
\begin{equation}
\label{eqCs}
\begin{aligned}
C_{z1}(\cdot)  :=  2 \int^{b}_{a} \kappa_{xx}(x)^2 dx,  ~
C_{z2}(\cdot)  := \int^{b}_{a} \kappa(x)^2 dx, ~
C_{z3}(\cdot)  := 2   \int^{b}_{a} \kappa_{x}(x)^2 dx,
\end{aligned}
\end{equation}
\begin{equation} \label{eqomega}
\begin{aligned} 
\kappa(x) & := z(a)-2[z(b) - z(a)] \left( \frac{x-a}{b} \right)^3 
 + 3 [z(b) - z(a)] \left( \frac{x-a}{b} \right)^2.
\end{aligned}
\end{equation}
\end{lemma} 

\ifitsdraft
\begin{remark}
A simple approach to explicitly compute $\delta$, a lower bound of 
$\delta_o$,  and to numerically compute $\delta_o$ are described in Lemma \ref{lemlowbound}.
\end{remark}
\else
\begin{remark}
We can explicitly compute $\delta$,  a lower bound of 
$\delta_o$,  and numerically approach $\delta_o$.  This is omitted due to space limitation.  
\end{remark}
\fi

\begin{remark} \label{rem1}
According to \eqref{eqCs} and \eqref{eqomega},  we conclude that, for each  $i \in \{1,2,3\}$, 
$C_{zi}(z(a),z(b))  := a_{zi} z(a)^2 + b_{zi} z(b)^2 + c_{zi} z(a) z(b), $
where $(a_{zi},b_{zi},c_{zi})$  are constants obtained by integrating  the  polynomials $\kappa^2$, $\kappa_{x}^2$, and $\kappa_{xx}^2$ on the interval $[a,b]$.  It is also important to note that  the parameters $(a_{zi},b_{zi},c_{zi})$ depend only on the domain of $z$,  which is the interval $[a,b]$.
\end{remark}

At this point,  using Lemma \ref{lem4},  we translate the analysis of 
$\Sigma_2$,  which is an infinite-dimensional system,  into the analysis of a finite-dimensional system of differential inequalities.   

\begin{lemma} \label{lemfd}
Along the solutions to $\Sigma_2$ and under \eqref{eqBC1} and \eqref{eqDerich}, we have 
\begin{equation}
\label{eqLyapgenbis1}
\begin{aligned}
\dot{V}_1 & \leq  2\delta_1 V_1 + C_{w1} (u_1,u_2) + \delta_2  C_{w2}(u_1,u_2) \\ & + \lambda_1  C_{w3}(u_1,u_2)
  -  \frac{u_2^3 - u_1^3}{3} - u_2 w_{xxx}(Y) + u_1 w_{xxx}(0), 
\\ \dot{V}_2 & \leq  2 \delta_1 V_2 + C_{v1} (u_2,u_3) + \delta_2  C_{v2}(u_2,u_3) \\ & + \lambda_1  C_{v3}(u_2,u_3)
 -  \frac{u_3^3 - u_2^3}{3} - u_3 v_{xxx}(L)   + u_2 v_{xxx}(Y), 
\end{aligned}
\end{equation}
where $(\delta_1,\delta_2)$ are given in Lemma \ref{lem4}, and $\{C_{wi} \}^3_{i=1}$ and $\{C_{vi} \}^3_{i=1}$ are obtained as in Lemma \ref{lem4} while substituting $(a,b,z)$ therein by $(0,Y,w)$ and $(Y,L,v)$, respectively. 
\end{lemma}

\begin{remark} \label{rem8}
The control action  $u_2 = w(Y) = v(Y)$ will be helpful to handle the boundary terms at $x = Y$ that appear in \eqref{eqLyapgenbis1}.    Without this additional control action, the problem is hard to solve when $\lambda_1 \geq 4 \pi^2$.
\end{remark}

\ifitsdraft
 Our  approach to control $\Sigma_2$  consists in the controlling the system of differential inequalities in \eqref{eqLyapgenbis1} in two steps.   When $t \in I_1 := \bigcup^{\infty}_{k = 1} [t_{2k-1},t_{2k})$,  the inputs are designed, based on the knowledge of the state variable $w$,  to stabilize $V_1$.  However, it is important to guarantee an appropriate behavior for $V_2$ during this time.  The same reasoning applies when $t \in I_2 := \bigcup^{\infty}_{k = 1} [t_{2k},t_{2k+1})$ mutatis mutandis. 
\fi

\section{Main Result 1:  Control at $x = 0$ and $x = L$} \label{Sec.DBC1}

In this section, we let $u_2 = 0$; hence, \eqref{eqLyapgenbis1} becomes
\begin{equation*} 
\begin{aligned}
\dot{V}_1  \leq & ~  2 \delta_1 V_1 + C_{w1}(u_1,0) + \delta_2 C_{w2}(u_1,0) +  \lambda_1 C_{w3}(u_1,0) + u_1 w_{xxx}(0) +  u^3_1/3, 
\\
\dot{V}_2  \leq & ~  2 \delta_1 V_2 + C_{v1}(0,u_3) + \delta_2 C_{v2}(0,u_3)  + \lambda_1 C_{v3}(0,u_3)  - u_3 v_{xxx}(L) - u^3_3/3.  
\end{aligned}
\end{equation*}
Next, using Remark \ref{rem1}, we obtain 
\begin{equation*} 
\begin{aligned}
\dot{V}_1  \leq & ~  2 \delta_1 V_1 + a_{w1} u^2_1 + \delta_2 a_{w2}u^2_1 +  \lambda_1 a_{w3} u^2_1 + u_1 w_{xxx}(0) +  u^3_1/3, 
\\
\dot{V}_2  \leq & ~  2 \delta_1 V_2 + b_{v1} u^2_3 + \delta_2 b_{v2}u^2_3 + \lambda_1  b_{v3}u^2_3  - u_3 v_{xxx}(L) - u^3_3/3.  
\end{aligned}
\end{equation*}

\subsection{Control Design}

\begin{itemize}
\item When $t \in I_1$, we measure $w([0,Y], t)$ and we choose 
$(u_1,u_3)$ so that $u_3  = 0 $ and
\begin{equation}
\label{eqdesign1}
\begin{aligned}
 \hspace{-0.4cm}  \frac{u^3_1}{3} + & (a_{w1}  + \delta_2 a_{w2} + \lambda_1 a_{w3} )  u_1^2  
 + u_1 w_{xxx}(0) 
  \leq  -(\alpha_1 + 2 \delta_1) V_1,
\end{aligned}
\end{equation}
for some $\alpha_1 > 0$.  Hence, we have
\begin{equation}
\label{eqSysI1}
\begin{aligned}
\dot{V}_1  \leq  - \alpha_1 V_1
\quad \text{and} \quad 
\dot{V}_2  \leq  2 \delta_1 V_2.
\end{aligned}
\end{equation}

\item When $t \in I_2$, we measure 
$u([Y,L],t)$ and we choose $(u_2,u_3)$ so that $u_1  = 0$ and
\begin{equation} \label{eqdesign2}
\begin{aligned} 
 \hspace{-0.3cm} - \frac{u^3_3}{3}   + & (b_{v1} + \delta_2 b_{v2}  + \lambda_1 b_{v3}) u_3^2 - u_3 v_{xxx}(L)   \leq  - (\alpha_2 + 2 \delta_1) V_2,
\end{aligned}
\end{equation}
for some $\alpha_2 > 0$.  Hence,  we obtain
\begin{equation}
\label{eqSysI2}
\begin{aligned}
\dot{V}_1  \leq  2 \delta_1 V_1 \quad \text{and} \quad 
\dot{V}_2  \leq  - \alpha_2 V_2.
\end{aligned}
\end{equation}
\end{itemize}

In the following lemma,  we show how to design $u_1$ and $u_3$ to satisfy \eqref{eqdesign1} and \eqref{eqdesign2},  respectively.  
\begin{remark}
Note that \eqref{eqdesign1} involves $w_{xxx}(0)$, which is not guaranteed to remain bounded.   Hence, it is important to design 
$u_1 := \kappa_1 (V_1,w_{xxx}(0))$  with $ w_{xxx}(0) \mapsto \kappa_1(V_1,w_{xxx}(0))$ globally bounded.  Similarly, \eqref{eqdesign2} involves $v_{xxx}(L)$.  Hence, it is important to design  $u_3 := \kappa_3(V_2,v_{xxx}(L))$  with $ v_{xxx}(L) \mapsto \kappa_3(V_2,v_{xxx}(L))$ globally bounded.  
\end{remark}

\begin{lemma} \label{lemdesign}
To satisfy \eqref{eqdesign1},  we take 
\begin{equation*}
\begin{aligned}
u_1 := \kappa_1  (V_1, w_{xxx}(0)) := 
\left\{ \begin{matrix}
  - \text{sign} (w_{xxx}(0))  V_1 & \text{if} ~ |w_{xxx}(0)| \geq l_1(V_1) \\
k_1(V_1) & \text{otherwise},
\end{matrix} \right.
\end{aligned}
\end{equation*}
where $k_1$ is such that
\begin{equation}
\label{eqx1}
\begin{aligned} 
k^3_1 & +3 (a_{w1}  + \delta_2 a_{w2} + \lambda_1 a_{w3} + 1)  k_1^2 + 
3 l_1(V_1)^2   \leq  -3(\alpha_1 + 2 \delta_1) V_1,
\end{aligned}
\end{equation}
and $l_1(V_1) :=  V_1^2/3 + (a_{w1}  + \delta_2 a_{w2})  V_1  + (\alpha_1 + 2 \delta_1)$.

Similarly,  to satisfy  \eqref{eqdesign2},  we take 
\begin{equation*}
\begin{aligned}
u_3 := \kappa_3  (V_2, v_{xxx}(L)) := 
\left\{ \begin{matrix}
  - \text{sign} (v_{xxx}(L))  V_2 & \text{if} ~ |v_{xxx}(L)| \geq l_3(V_2) \\
k_3(V_2) & \text{otherwise},
\end{matrix} \right.
\end{aligned}
\end{equation*}
where $k_3$ is such that
\begin{equation}
\label{eqx2}
\begin{aligned} 
k^3_3 & +3 (b_{v1}  + \delta_2 b_{v2} + \lambda_1 b_{v3} + 1)  k_3^2 + 3 l_3(V_2)^2   \leq  -3(\alpha_2 + 2 \delta_1) V_2,
\end{aligned}
\end{equation}
and $l_3(V_2) :=  V_2^2/3 + (b_{v1}  + \delta_2 b_{v2})  V_2  + (\alpha_2 + 2 \delta_1)$.
\end{lemma}

\begin{remark}
Note that \eqref{eqx1} and \eqref{eqx2} always admit a solution $k_1$ and $k_3$ function of $V_1$ and $V_2$, respectively.  
For example,   one can take $k_1$ as a second-order polynomial of $V_1$  with strictly negative coefficients that are sufficiently large. 
\end{remark}

\subsection{$\mathcal{L}^2$ Exponential Stability}

Let $\Sigma^{cl}_1$  be the system obtained from  $\Sigma_1$ when   \eqref{eqDerich1} holds,   $u_2 = 0$,  $(u_1,u_3) = (\kappa_1,0)$ on $I_1$, and  $(u_1,u_3) = (0, \kappa_3)$ on $I_2$.  
In this section,  we show how to find positive constants $\alpha_1$ and $\alpha_2$ such that the trivial solution to $\Sigma^{cl}_1$ is  
$\mathcal{L}^2$ globally exponentially stable.

\begin{definition} 
The trivial solution to $\Sigma^{cl}_1$ is $\mathcal{L}^2$-GES  if there exist $\gamma$, $\kappa > 0$ such that, for each solution $u$ to $\Sigma^{cl}_1$,  we have  $||u(t_o + t)||  \leq \kappa e^{- \gamma t} ||u(t_o)||$ for all $(t_o,t) \in \mathbb{R}_{\geq 0} \times   \mathbb{R}_{\geq 0}$. 
\end{definition}
According to the proposed approach,  we establish the 
$\mathcal{L}_2$-GES for $\Sigma^{cl}_1$ by showing, for an appropriate choice of $(\alpha_1,\alpha_2)$,  GES of the origin for the switched system
\begin{equation*} 
\begin{aligned}
\Sigma_3 : \left\{
\begin{matrix}
\begin{matrix}
\dot{V}_1 = & - \alpha_1 V_1
\\
\dot{V}_2 = & 2 \delta_1 V_2
\end{matrix}
\qquad  t \in I_1,
\\   ~~ \\ 
\begin{matrix}
\dot{V}_1 = & 2 \delta_1 V_1 
\\
\dot{V}_2 = & - \alpha_2 V_2
\end{matrix}
\qquad  t \in I_2,
\end{matrix}
\right.  \qquad  (V_1,V_2) \in \mathbb{R}_{\geq 0} \times \mathbb{R}_{\geq 0}.
\end{aligned}
\end{equation*}

\begin{theorem} \label{thm1}
Consider system $\Sigma_1$ under the sensing scenario in Section \ref{Sec.sens} and the boundary conditions in \eqref{eqDerich1}.  Assume that,  for some $\bar{T}_1$,  $\bar{T}_2>0$, Assumption \ref{ass1} holds.  
Furthermore, we let  $u_2 = 0$, $(u_1,u_3) = (\kappa_1,0)$ on $I_1$,  and  $(u_1,u_3) = (0, \kappa_3)$ on $I_2$, where $\kappa_1,\kappa_3$ come from Lemma \ref{lemdesign}. 
Then,  for $(\alpha_1,\alpha_2)$ satisfying 
\begin{equation}  \label{eq.condi}
\begin{aligned}
 \alpha_1    >  \frac{2\delta_2 \bar{T}_2}{\bar{T}_1} 
\quad \text{and} \quad 
 \alpha_2  >   \frac{2\delta_2 \bar{T}_1}{\bar{T}_2},
\end{aligned}
\end{equation}
  the trivial solution to the closed-loop system is $\mathcal{L}^2$-GES.
\end{theorem}

 \ifitsdraft
\begin{proof}
 According to the proposed framework, we analyze the 
$\mathcal{L}_2$-GES of $\Sigma_1$ by analyzing the same property for $\Sigma_2$. Furthermore,  since \eqref{eqdesign1} holds on $I_1$ and \eqref{eqdesign2} holds on $I_2$, it is enough to show GES of the switched system $\Sigma_3$.  Namely,  we will show the existence of $\gamma, \kappa > 0$ such that, for every solution $(V_1,V_2)$ to $\Sigma_3$ starting from $(V_1(t_o),V_2(t_o))$, we have 
\begin{equation} \label{eqGES}
\begin{aligned} 
W(t_o + t) & := V_1(t_o + t)  + V_2(t_o + t) 
  \leq \kappa e^{- \gamma t} \left(  V_1(t_o) +  V_2(t_o) \right) \quad \forall t \geq t_o.        
\end{aligned}
\end{equation}
To do so, we introduce the sequence 
$\{T_i\}^{\infty}_{i = 0}$ with $T_o := 0$ and $T_i = t_{2i+1}$, and we note that 
$$ 
T_{i+1} - T_i = \bar{T}_1 + \bar{T_2} \qquad \forall i \in \{0,1,2,...\}. 
$$

Furthermore, along each solution $(V_1,V_2)$ to $\Sigma_3$, we have
\begin{align*} 
& W(T_{i+1})  := 
V_1(T_{i+1}) + V_2(T_{i+1}) \\ 
&  = e^{- \alpha_1 \bar{T}_1 + 2\delta_2 \bar{T}_2} V_1(T_{i})  + e^{- \alpha_2 \bar{T}_2 + 2\delta_2 \bar{T}_1} V_2(T_{i})
\\ & \leq  \max \left\{ e^{- \alpha_1 \bar{T}_1 + 2\delta_2 \bar{T}_2},  e^{- \alpha_2 \bar{T}_2 + 2\delta_2 \bar{T}_1} \right\} \left(V_1(T_i) + V_2(T_i)  \right) 
\\ &
\leq  \max \left\{ e^{- \alpha_1 \bar{T}_1 + 2\delta_2 \bar{T}_2},  e^{- \alpha_2 \bar{T}_2 + 2\delta_2 \bar{T}_1} \right\} W(T_i).
\end{align*} 
Now, we choose the constants $\alpha_1>0$ and $\alpha_2>0$ such that 
\begin{equation}
\label{eq.condi+}
\begin{aligned}
- \alpha_1 \bar{T}_1 + 2\delta_2 \bar{T}_2  <0
\quad \text{and} \quad
- \alpha_2 \bar{T}_2 + 2\delta_2 \bar{T}_1  <0.
\end{aligned}
\end{equation}
As a consequence, we obtain 
\begin{align*}
 W(T_{i+1}) \leq  e^{- \min \left\{ (\alpha_1 \bar{T}_1 - 2\delta_2 \bar{T}_2),  (\alpha_2 \bar{T}_2 - 2\delta_2 \bar{T}_1) \right\} } W(T_i).
\end{align*} 
Hence, by taking 
$$ \beta :=\frac{ \min \left\{ (\alpha_1 \bar{T}_1 - 2\delta_2 \bar{T}_2), (\alpha_2 \bar{T}_2 - 2\delta_2 \bar{T}_1) \right\} }{\bar{T}_1 + \bar{T_2}}, $$
we conclude that
\begin{align*} 
W(T_{i+1}) & \leq e^{ - \beta (T_{i+1} - T_i ) } W(T_{i}) \qquad \forall i \in \{0,1,2,...\}. 
\end{align*} 
Next, we verify \eqref{eqGES} while showing that $\beta$ is a convergence rate. Indeed, note that, for each $t \in [T_o,T_1]$, we have 
\begin{align*}
W(t) \leq ~ & e^{2\delta_1 (\bar{T}_1 + \bar{T}_2)} W(T_o) 
\leq  e^{(2\delta_1 + \beta) (\bar{T}_1 + \bar{T}_2)}  e^{-\beta(t - T_o)}  W(T_o).
\end{align*}
Next, for each $t \in [T_i,T_{i+1}]$ such that $i \in \{1,2,...\}$, we have 
\begin{align*}
W(t) \leq ~ & e^{2\delta_1 (\bar{T}_1 + \bar{T}_2)} W(T_i) 
\\
\leq ~ & 
e^{2\delta_1 (\bar{T}_1 + \bar{T}_2)}  e^{-\beta(T_i - T_o)}  W(T_o)     
\\
\leq ~ & e^{(2\delta_1 + \beta) (\bar{T}_1 + \bar{T}_2)} e^{-\beta(t - T_i)} e^{-\beta(T_i - T_o)}  W(T_o) 
\\
\leq ~& e^{(2\delta_1 + \beta) (\bar{T}_1 + \bar{T}_2)}  e^{-\beta(t - T_o)}  W(T_o).
\end{align*} 
Hence, \eqref{eqGES} holds with $\gamma = \beta$  and 
$\kappa = e^{(2\delta_1 + \beta) (\bar{T}_1 + \bar{T}_2)}$.
\end{proof}
\fi

\section{Main Result 2 : Control at $x = 0$ and $x = Y$} \label{Sec.DBC2}

In this section,  we let $u_3 = 0$. As a result,   \eqref{eqLyapgenbis1} reduces to
\begin{equation} \label{eqLyap2bis}
\begin{aligned}
\hspace{-0.4cm} \dot{V}_1  \leq &~ 2 \delta_1 V_1 - \frac{u^3_2}{3} - u_2 w_{xxx}(Y) + \frac{u^3_1}{3}  + u_1 w_{xxx}(0) 
\\ & + 
C_{w1}(u_1,u_2)  + C_{w2}(u_1,u_2) \delta_2 + C_{w3}(u_1,u_2) \lambda_1,  
\\
\hspace{-0.4cm} \dot{V}_2 \leq &~ 2 \delta_1 V_2 + \frac{u^3_2}{3}  +  u_2 v_{xxx}(Y)
\\ &
+ C_{v1}(u_2,0)  + C_{v2}(u_2,0)  \delta_2 + 
C_{v3}(u_2,0)  \lambda_1.  
\end{aligned}
\end{equation}

\subsection{Control Design}

When $t \in I_1$,  we set  $u_2 = 0$ and choose $u_1$ such that  \eqref{eqdesign1} holds.  To obtain
\begin{equation}
\label{eqonI1}
\begin{aligned}
\dot{V}_1   \leq  - \alpha_1 V_1
\quad  \text{and}  \quad 
 \dot{V}_2  \leq  2 \delta_1 V_2.
\end{aligned}
\end{equation}

When $t \in I_2$,  we note that the first inequality in \eqref{eqLyap2bis}  involves the term $\left[ u_1 w_{xxx}(0) \right]$.   Note that $w_{xxx}(0)$ is unknown on the interval $I_2$.  Hence, before designing $(u_1,u_2)$, we introduce the following lemma.

\begin{lemma} \label{lemwxxx}
Consider $\Sigma_2$ under \eqref{eqBC1} and \eqref{eqDerich}. Then
\begin{align*}
w_{xxx}(0) =  w_{xxx}(Y) +  \frac{u^2_2 - u^2_1}{2}  
+ \dot{\gamma},
\end{align*}
where $\gamma := \left(\int^{Y}_{0} w(x) dx \right)$.
\end{lemma}

Using Lemma \ref{lemwxxx},  \eqref{eqLyap2bis} becomes 
\begin{equation} \label{eqLyap3}
\begin{aligned}
 \dot{V}_1 & \leq  - \frac{u^3_2}{3} + (u_1- u_2) w_{xxx}(Y) -  
 \frac{u^3_1}{6} 
\\ & + C_{w1}(u_1,u_2) +  \frac{u_1 u^2_2}{2} + 2 \delta_1 V_1 
+ \delta_2 C_{w2}(u_1,u_2) 
+ \lambda_1 C_{w3}(u_1,u_2) + u_1 \dot{\gamma},  
\\ 
\dot{V}_2 & \leq \frac{u^3_2}{3} + C_{v2}(u_2,0) \delta_2 + C_{v3}(u_2,0) \lambda_1  + C_{v1}(u_2,0)  + u_2 v_{xxx}(Y) + 2 \delta_1 V_2.  
\end{aligned}
\end{equation}
Now, we introduce constants $B>0$ and $C>0$ such that
\begin{align*} 
B u_2^2 & := C_{v3}(u_2,0) \lambda_1 + C_{v2}(u_2,0) \delta_2 + C_{v2}(u_2,0),
\\
C u_2^2 & := C_{w1}(u_2,u_2)  + \delta_2 C_{w2}(u_2,u_2) + \lambda_1 C_{w3}(u_2,u_2).
\end{align*}
Hence,  on the interval $I_2$,  we propose to choose $(u_1,u_2)$ 
such that $u_1 =  u_2$ and 
\begin{equation} \label{eqdesign2bis}
\begin{aligned}
\frac{u^3_2}{3} + B u_2^2  + u_2 v_{xxx}(Y) &  \leq -\alpha_2 V_2^3.
\end{aligned}
\end{equation}

As a consequence, we obtain,  for almost all $t \in I_2$,
\begin{equation}
\label{eqonI2}
\begin{aligned}
\dot{V}_1  \leq 2 \delta_1 V_1  + u_2 \dot{\gamma} +  C u_2^2
\quad \text{and} \quad 
\dot{V}_2 \leq - \alpha_2 V_2^3 + 2 \delta_1 V_2.
\end{aligned}
\end{equation}

Now,   we propose to find 
$ u_2 := \kappa_2 \left( \alpha_2^{\frac{1}{3}} V_2, v_{xxx}(Y) \right)$ so that  both \eqref{eqdesign2bis} and the following property hold. 
  
 \begin{property} \label{ass2}
There exists $P > 0$ such that
\begin{align} \label{asscond1}
 \bigg|\kappa_2 \left( \alpha_2^{\frac{1}{3}} V_2,v_{xxx}(Y) \right)\bigg| \leq  P \alpha_2^{\frac{1}{3}} V_2 \qquad \forall v_{xxx}(Y) \in \mathbb{R},
\end{align}
and,  for almost all $t \geq t_o \geq 0$,  we have 
\begin{align} \label{asscond2} 
 \frac{d}{dt} 
\kappa_2 \left( \alpha_2^{\frac{1}{3}} V_2(t_o),v_{xxx}(Y,t) \right)   = 0.
 \end{align} 
\end{property}    

\begin{lemma} \label{lemdesu2}
Property \ref{ass2} and  \eqref{eqdesign2bis}  hold for 
\begin{equation*}
\begin{aligned}
  \kappa_2  \left( \alpha_2^{\frac{1}{3}} V_2,  v_{xxx}(Y) \right) := 
\left\{ \begin{matrix}
  - \text{sign} (v_{xxx}(Y)) \alpha_2^{\frac{1}{3}} V_2 & \text{if} ~ |v_{xxx}(Y)| \geq 2  \alpha_2^{\frac{2}{3}} V_2^2 \\
- \beta \alpha_2^{\frac{1}{3}} V_2 & \text{otherwise},
\end{matrix} \right.
\end{aligned}
\end{equation*}
where  $\beta > 0$ is chosen so that $ -\beta^3 + 6 \beta + 3 \leq 0$. 
\end{lemma}

 \subsection{$\mathcal{L}^2$-Stability Analysis} 

Let $\Sigma^{cl}_1$  be the system obtained from  $\Sigma_1$ when   \eqref{eqDerich1} holds,   $u_3 = u_2 = 0$,  $u_1 = \kappa_1(V_1,w_{xxx}(0))$ on $I_1$, and $u_3 = 0$ and $u_1 = u_2 = \kappa_2  \left( \alpha_2^{\frac{1}{3}} V_2(t_{2k}),  v_{xxx}(Y) \right)$ on each interval $[t_{2k}, t_{2k+1} ] \subset I_2$.  Recall that,  by definition of $\kappa_1$ and $\kappa_2$,  $\Sigma^{cl}_1$ includes $(\alpha_1,\alpha_2)$ as free design parameters.  Next,  inspired by \cite{teel1999semi}, we introduce some useful  $\mathcal{L}^2$-semi-global and $\mathcal{L}^2$-practical-stability notions.

 The trivial solution to $\Sigma^{cl}_1$ is $\mathcal{L}^2$ practically semi-globally attractive ($\mathcal{L}^2$-PSGA) if,  for each $\beta > \epsilon > 0$, there exists  $\alpha_2^{\star} > 0$ such that, for each $\alpha_2 \geq \alpha^\star_2$, there exists $\alpha^\star_1$ such that, for each $\alpha_1 \geq \alpha^\star_1$,   every solution $u$ to $\Sigma^{cl}_1$ with  $|| u(t_o) || \leq \beta$,  there exists $T>0$ such that  $||u(t_o + T)||  \leq \epsilon$.

 $\Sigma^{cl}_1$ is $\mathcal{L}^2$ semi-globally bounded ($\mathcal{L}^2$-SGB) if,  for each $\beta>0$,  there exists exists $\gamma>0$ and 
$\alpha_2^{\star} > 0$ such that, for each $\alpha_2 \geq \alpha^\star_2$, there exists $\alpha^\star_1$ such that, for each $\alpha_1 \geq \alpha^\star_1$, we have,  
for every solution $u$ to $\Sigma^{cl}_1$ with $||u(t_o)|| \leq \beta$,  $||u(t_o + t)|| \leq  \gamma$ for all $t \geq  0$.   Furthermore,  $\Sigma^{cl}_1$ is $\mathcal{L}^2$ semi-globally ultimately bounded ($\mathcal{L}^2$-SGUB) if there exists $\gamma>0$ such that, for each $\beta>0$,  there exists $\alpha_2^{\star} > 0$ such that, for each $\alpha_2 \geq \alpha^\star_2$, there exists $\alpha^\star_1$ such that, for each $\alpha_1 \geq \alpha^\star_1$,  for every solution $u$ to $\Sigma^{cl}_1$ with $||u(t_o)|| \leq \beta$,  there exists $T>0$ such that  $||u(t_o + t)|| \leq  \gamma$ for all $t \geq  T$.  

The trivial solution to $\Sigma^{cl}_1$ is $\mathcal{L}^2$ practically stable ($\mathcal{L}^2$-PS) if there exists  $\kappa \in \mathcal{K}$ such that,  for each $\epsilon > 0$,  there exists  $\alpha_2^{\star} > 0$ such that, for each $\alpha_2 \geq \alpha^\star_2$, there exists $\alpha^\star_1$ such that, for each $\alpha_1 \geq \alpha^\star_1$,   we have $||u(t)|| \leq \kappa \left( ||u(t_o)|| \right) + \epsilon$ for $t \geq t_o$.

\begin{remark}
Note that to guarantee $\mathcal{L}_2$-practical semi-global asymptotic stability,  we need to guarantee $\mathcal{L}_2$-PSGA,  $\mathcal{L}_2$-SGB,  and $\mathcal{L}_2$-PS.   However,  in our case,  we will be able to show only $\mathcal{L}_2$-PSGA,  $\mathcal{L}_2$-SGB,  and $\mathcal{L}_2$-SGUB. 
\end{remark}

\begin{theorem} \label{thm2}
Consider system $\Sigma_1$ under the sensing scenario in Section \ref{Sec.sens} and the boundary conditions in \eqref{eqDerich1}.  Assume that, for some $\bar{T}_1, \bar{T}_2>0$, Assumption \ref{ass1} holds.   Furthermore,  we  let  $u_3 = u_2 = 0$,  $u_1 = \kappa_1(V_1, w_{xxx}(0))$ on $I_1$,  and $u_1 = u_2 = \kappa_2 \left( \alpha_2^{\frac{1}{3}} V_2(t_{2k}),  v_{xxx}(Y) \right)$ on each interval $[t_{2k}, t_{2k+1} ] \subset I_2$ and $u_3 = 0$ on $I_2$, where $(\kappa_1,\kappa_2)$ come from Lemmas \ref{lemdesign} and \ref{lemdesu2}, respectively.   Then,  the trivial solution to $\Sigma^{cl}_1$ is $\mathcal{L}^2$-PSGA and  $\Sigma^{cl}_1$ is $\mathcal{L}^2$-SGB and  $\mathcal{L}^2$-SGUB. 
\end{theorem}

\ifitsdraft

\subsection{Proof of Theorem \ref{thm2}}

In this section, we illustrate the key intermediate steps  to prove Theorem \ref{thm2}.   Furthermore,  we point out the particular step preventing us from guaranteeing the $\mathcal{L}^2$-PS.    

According to \eqref{eqonI1} and \eqref{eqonI2},  to analyze the 
$\mathcal{L}^2$ properties of $\Sigma^{cl}_1$ in Theorem \ref{thm2},   it is enough to analyze the same properties for the switched differential inequality: 
\begin{equation*}
\begin{aligned}
\Sigma^b_3: 
\left\{
\begin{matrix}
\begin{matrix}
\dot{V}_1 =  - \alpha_1 V_1
\\
\dot{V}_2 =  2 \delta_1 V_2
\end{matrix}
\qquad \quad \forall t \in I_1
\\
~~
\\
\begin{matrix}
\dot{V}_1 \leq  2 \delta_1 V_1 + \kappa_2 \dot{\gamma}  + C \kappa_2^2
\\
\dot{V}_2 \leq   - \alpha_2 V_2^3 + 2 \delta_1 V_2 
\end{matrix}
\qquad \forall t \in I_2,
\end{matrix}
\right.
\end{aligned}
\end{equation*}
where $(V_1,V_2) \in \mathbb{R}_{\geq 0} \times \mathbb{R}_{\geq 0}$ and 
\begin{align} \label{eqV1gamma}
V_1 := \frac{1}{2}\int^{Y}_{0} w(x)^2 dx \quad \text{and} \quad \gamma := \int^{Y}_{0} w(x) dx. 
\end{align}

Now,  we introduce the following lemma.
\begin{lemma} \label{lem8}
Given $\delta_1 \in \mathbb{R}$ and a positive constant  
$\alpha_2$. Then, there exists $b_o > 0$ such that,  for each $\pi \geq 0$, 
\begin{align} \label{eqlem81}  
\pi^3 -  2 \delta_1  \alpha_2^{-\frac{1}{3}}  \pi  \geq  \pi - b_o.
\end{align}
\end{lemma}

According to Lemma \ref{lem8},  to analyze the switched differential inequality $\Sigma^b_3$,   it is enough to analyze 
\begin{equation*}
\begin{aligned}
\Sigma_4: 
\left\{
\begin{matrix}
\begin{matrix}
\dot{V}_1 =  - \alpha_1 V_1
\\
\dot{V}_2 =  2 \delta_1 V_2
\end{matrix}
\qquad \quad \forall t \in I_1,
\\
~~
\\
\begin{matrix}
\dot{V}_1 \leq  2 \delta_1 V_1 + u_2 \dot{\gamma}  + C u_2^2
\\
\dot{V}_2 \leq   - \alpha_2^{\frac{1}{3}} V_2 + b_o
\end{matrix}
\qquad \forall t \in I_2,
\end{matrix}
\right.
\end{aligned}
\end{equation*}
where $(V_1,V_2) \in \mathbb{R}_{\geq 0} \times \mathbb{R}_{\geq 0}$  and $V_1$ and $\gamma$ are related via \eqref{eqV1gamma}.  

 Next,  we introduce the scalar (infinite-dimensional) differential inequality  
\begin{align*}
\Sigma_5 : \dot{V}_1  \leq 2 \delta_1 V_1 + u_2(t,t_o) 
\dot{\gamma} + C u_2(t,t_o)^2,
\end{align*}
where $u_2 : \mathbb{R}_{\geq 0} \times  \mathbb{R}_{\geq 0}  \rightarrow \mathbb{R}$ is an external signal,
$\delta_1 \in \mathbb{R}$ and $C>0$ are constants,  and $V_1$ and $\gamma$ are related via \eqref{eqV1gamma}.  

The following result establishes an upperbound on the solution $V_1$ to $\Sigma_5$.  
 
\begin{lemma} \label{lem6new}
Consider system $\Sigma_5$ and assume that $u_2$ is absolutely continuous.   Let $V_1$ be a solution to $\Sigma_5$ starting from $V_1(t_o)$ at $t_o \geq 0$. Then, for all $t \geq t_o$,  we have   
\begin{align}  \label{eqV1++}
 V_1(t) \leq ~ & \alpha(t,t_o) + 
\int^{t}_{t_o} \alpha(\tau,t_o) \beta(t,\tau, t_o) 
e^{ \int^{t}_{\tau}  \beta(t,s,t_o) d s } d\tau, 
\end{align}
where 
\begin{align*}
\beta(t,\tau,t_o)  := 2 \bigg|
\frac{d}{d\tau} g(t,\tau,t_o) \bigg|,  \quad 
 g(t,\tau,t_o)  :=  e^{2\delta_1(t-\tau)} u_2(\tau,t_o),  
\end{align*}
\begin{align*}
\alpha(t,t_o) & :=  2 e^{2\delta_1 (t-t_o)} V_1(t_o) 
+  2 C \int^{t}_{t_o} u_2(\tau,t_o) g(t,\tau,t_o) d\tau 
\\ &
+ 2 u_2(t,t_o)^2 + 2 |g(t,t_o,t_o)| 
\sqrt{V_1(t_o)}  + \frac{1}{2}
 \int^{t}_{t_o}  \beta(t,\tau,t_o)   d\tau.
 \end{align*}
\end{lemma}

\begin{proof}
By solving $\Sigma_5$, we obtain 
\begin{align*}
V_1(t) = & ~ e^{2\delta_1 (t-t_o)} V_1(t_o) +  
C \int^{t}_{t_o} u_2(\tau,t_o)^2 e^{2\delta_1(t-\tau)} d\tau 
 + \int^{t}_{t_o} e^{2\delta_1 (t-\tau)} u_2(\tau,t_o) 
\dot{\gamma}(\tau) d\tau.
\end{align*}
Using the integration by part, we obtain 
\begin{align*}
V_1(t) = ~ & e^{2\delta_1 (t-t_o)} V_1(t_o) +  
C \int^{t}_{t_o} u_2(\tau,t_o)^2 e^{2\delta_1(t-\tau)} d\tau 
\\ & 
+ \left[e^{2\delta_1 (t-\tau)} u_2(\tau,t_o)  
\gamma(\tau)\right]^{t}_{t_o}
 -\int^{t}_{t_o} \frac{d}{d\tau} 
\left[e^{2\delta_1 (t-\tau)} u_2(\tau,t_o) \right] \gamma(\tau) d\tau.
\end{align*}
Hence, we obtain
\begin{align*}
 V_1(t) = ~ & e^{2\delta_1 (t-t_o)} V_1(t_o) + 
C \int^{t}_{t_o} u_2(\tau,t_o)^2 e^{2\delta_1(t-\tau)} d\tau 
\\ & +
\left[u_2(t,t_o) \gamma(t)- e^{2\delta_1 (t-t_o)} u_2(t_o,t_o)  
\gamma(t_o) \right] 
 \\ & -
\int^{t}_{t_o} \frac{d}{d\tau} 
\left[e^{2\delta_1(t-\tau)} u_2(\tau,t_o) \right] \gamma(\tau) d\tau.
\end{align*}
Next, we let $ g(t,\tau,t_o) :=   e^{2\delta_1(t-\tau)} u_2(\tau,t_o)$,
to obtain 
\begin{align*}
 V_1(t) = ~ & e^{2\delta_1 (t-t_o)} V_1(t_o) + 
C \int^{t}_{t_o} u_2(\tau,t_o) g(t,\tau,t_o) d\tau 
\\ & +
\left[ g(t,t,t_o) \gamma(t)- g(t,t_o,t_o)  \gamma(t_o) \right] 
 -
\int^{t}_{t_o} \frac{d}{d\tau} 
\left[g(t,\tau,t_o) \right] \gamma(\tau) d\tau. 
\end{align*}
Now, we upperbound the term 
$$ \left[ g(t,t,t_o) \gamma(t)- g(t,t_o,t_o)  \gamma(t_o) \right]$$ 
using the inequality
\begin{equation}
\label{eqrough}
\begin{aligned} 
g(t,t,t_o) \gamma(t)-  g(t,t_o,t_o)  \gamma(t_o)  \leq   \left[ g(t,t,t_o)^2 + \frac{1}{4} \gamma(t)^2 \right]   + \big|g(t,t_o,t_o) \big| 
\sqrt{V_1(t_o)}. 
\end{aligned}
\end{equation}
As a result,  we obtain  
\begin{equation*}
\begin{aligned} 
 V_1(t) &  \leq 
e^{2\delta_1 (t-t_o)} V_1(t_o) + 
C \int^{t}_{t_o} u_2(\tau,t_o) g(t,\tau,t_o) d\tau 
\\ & +
\left[ g(t,t,t_o)^2 + \frac{1}{4} \gamma(t)^2 \right] + |g(t,t_o,t_o)| 
\sqrt{V_1(t_o)}  \\ & +
\int^{t}_{t_o} \bigg|
\frac{d}{d\tau} g(t,\tau,t_o) \bigg|  \left[ \frac{1}{2} 
|\gamma(\tau)|^2  + \frac{1}{2} \right]  d\tau.
\end{aligned}
\end{equation*}
Next, we use the fact that 
$\frac{1}{2} \gamma^2 \leq V_1$, we obtain 
\begin{align*}
\frac{1}{2} V_1(t) & \leq 
e^{2\delta_1 (t-t_o)} V_1(t_o) + 
C \int^{t}_{t_o} u_2(\tau,t_o) g(t,\tau,t_o) d\tau 
 + g(t,t,t_o)^2 \\ & + |g(t,t_o,t_o)| 
\sqrt{V_1(t_o)}  + 
\frac{1}{2} \int^{t}_{t_o} \bigg|
\frac{d}{d\tau} g(t,\tau,t_o) \bigg|   d\tau 
 + \int^{t}_{t_o} \bigg|
\frac{d}{d\tau} g(t,\tau,t_o) \bigg| 
 V_1(\tau) d\tau.
\end{align*}
Next, we take 
\begin{align*}
 \alpha(t,t_o) &  := 2 e^{2\delta_1 (t-t_o)} V_1(t_o) 
 + 2 C \int^{t}_{t_o} u_2(\tau,t_o) g(t,\tau,t_o) d\tau 
\\ & +
2 g(t,t,t_o)^2 + 2 |g(t,t_o,t_o)| 
\sqrt{V_1(t_o)}   +  \int^{t}_{t_o} \bigg|
\frac{d}{d\tau} g(t,\tau,t_o) \bigg|   d\tau
\end{align*}
and $\beta(t,\tau,t_o) := 2 \bigg|\frac{d}{d\tau} g(t,\tau,t_o) \bigg|$. 
As a result, we have 
\begin{align*}
V_1(t) \leq \alpha(t,t_o) + \int^{t}_{t_o} \beta(t,\tau,t_o) V_1(\tau) 
d\tau.
\end{align*} 
Hence, using Lemma \ref{lem5-},  \eqref{eqV1++} follows.  
\end{proof}

\begin{remark} \label{rem5}
We note that $\alpha$ is an upper bound on the solution $V_1$ to the differential inequality  
\begin{align*}
 \dot{V}_1  \leq 2 \delta_1 V_1 + u_2(t) 
\dot{\gamma}.
\end{align*}
This upperbound is not tight enough to conclude that $ \lim_{t \rightarrow t_o} \alpha(t,t_o)  \leq \kappa (V_1(t_o))$,  for some class $\mathcal{K}$ function $\kappa$.  Indeed, we can see that 
$$ \lim_{t \rightarrow t_o} \alpha(t,t_o) = 2 V_1(t_o) + 2 u_2(t_o,t_o)^2 + 2 |u_2(t_o,t_o)| \sqrt{V_1(t_o)}.  $$
 This restrictive character is due to the way we upper-bounded the term 
 $$ \left[g(t,t,t_o) \gamma(t) - g(t,t_o,t_o) \gamma(t_o) \right]$$
  in the proof of Lemma \ref{lem6new}.  
 
A less na\"{i}f approach to upperbound the term $g(t,t) \gamma(t)- g(t,t_o)  \gamma(t_o)$ consists in using the following inequality 
\begin{equation}
\label{eqlessrough}
\begin{aligned} 
 g(t,t) \gamma(t)- g(t,t_o)  \gamma(t_o) & \leq 
  \gamma(t) [g(t,t) - g(t,t_o)] 
 + g(t,t_o) [ \gamma(t) - \gamma(t_o) ] 
\\   & \leq  \frac{V_1(t)}{2}  + |g(t,t) - g(t,t_o)|^2  + |g(t,t_o)| |\gamma(t) - \gamma(t_o)|.  
\end{aligned}
\end{equation}
Note that the right-hand side in \eqref{eqlessrough} converges to zero as $t \rightarrow t_o$. However,  we still need to upper bound the term 
$|\gamma(t) - \gamma(t_o)|$ by a function that goes to zero as 
$t \rightarrow t_o$. Furthermore, to pursue the computations after \eqref{eqlessrough}, the upper bound of the term $|\gamma(t) - \gamma(t_o)|$ needs to be function of $(t,t_o,V_1(t_o))$ only.  
Unfortunately, the latter is hard to reach. In particular, one could think of using the classical mean-value theorem, which gives 
$ |\gamma(t) - \gamma(t_o)| \leq \sup_{\tau \in [t_o,t]} |\dot{\gamma}(\tau)|(t-t_o)$.  Next, using the fact that 
\begin{align*}
 \dot{\gamma}  =  \int^{Y}_{0} \hspace{-0.2cm} w_t(x) dx  = & -  \int^{Y}_{0} \hspace{-0.2cm} w(x) w_x(x) dx - \lambda_1  \int^{Y}_{0} \hspace{-0.2cm} w_{xx}(x) dx   -  \int^{Y}_{0} \hspace{-0.2cm} w_{xxxx}(x) dx.
\end{align*}
Furthermore, assuming that $[t_o,t] \subset I_2$, we conclude that
$w(Y) = w(0)$ and $w_x(Y) = w_x(0) = 0$. Hence,
\begin{align*}
 \dot{\gamma}  = & -  \int^{Y}_{0} w_{xxxx}(x) dx  =  w_{xxx}(0)-  w_{xxx}(Y).
\end{align*}
As a result, we conclude that
\begin{align*}
\hspace{-0.4cm} |\gamma(t) - \gamma(t_o)| \leq \sup_{\tau \in [t_o,t]} \bigg|w_{xxx}(0,\tau)-  w_{xxx}(Y,\tau) \bigg|
(t-t_o). 
\end{align*}
Note that we are not making any boundedness assumption on 
$w_{xxx}$. Moreover,  we are not aware of any method allowing to conclude boundedness of $w_{xxx}$ without restricting the set of solutions to $\Sigma_2$. In conclusion, we are not able to find a more suitable upperbound of the term $g(t,t) \gamma(t)- g(t,t_o)  \gamma(t_o)$ than the one in \eqref{eqrough}.
\end{remark}

Next, we provide an explicit upperbound of a solution $V_1$ to 
$\Sigma_5$,  when  $u_2 = \kappa_2  \left( \alpha_2^{\frac{1}{3}} V_2(t_o),  v_{xxx}(Y) \right) $,  where $\kappa_2$ is introduced in Lemma \ref{lemdesu2},  and when,  for some $b_o>0$,  $V_2$ satisfies
\begin{align} \label{dotV2Lin} 
\dot{V_2}  \leq  - \alpha_2^{\frac{1}{3}} V_2 + b_o \qquad V_2  \geq b_o \alpha_2^{-\frac{1}{3}}.
\end{align} 

\begin{lemma} \label{lem7}
Consider system $\Sigma_5$  and assume that $u_2 :=   \kappa_2  \left( \alpha_2^{\frac{1}{3}} V_2(t_o),  v_{xxx}(Y) \right)  $, $V_2$ satisfies \eqref{dotV2Lin}, and $t \mapsto v_{xxx}(Y,t)$ is absolutely continuous. Then, there exists $M_1, M_2 : \mathbb{R}_{\geq 0} \rightarrow \mathbb{R}_{\geq 0}$ continuous and non-decreasing such that,  for all $t \in [t_o, t_o + T]$, we have 

\begin{equation}
\label{eqV1upbd}
\begin{aligned} 
 V_1(t)  \leq   \mathcal{V}_1 \left(T, V_1(t_o), \alpha_2^{\frac{1}{3}} V_2(t_o) \right),
\end{aligned}
\end{equation} 
 where 
\begin{align*} 
\mathcal{V}_1 \left(\cdot \right) & := T  \bar{\alpha}(t_o, T)  \beta_{1}(t_o, T)  e^{ \beta_{1}(t_o, T) T }   + \bar{\alpha}(t_o, T),
\\
  \bar{\alpha}(t_o, T) & :=  (2+P) e^{2\delta_1 T} V_1(t_o)   +  M_1(T)  \left[ \alpha_2^{\frac{2}{3}} V_2(t_o)^2  + \alpha_2^{\frac{1}{3}} V_2(t_o)  \right],
\\  
\beta_{1}(t_o, T) & :=  M_2(T)  \alpha_2^{\frac{1}{3}} V_2(t_o).
\end{align*} 
\end{lemma}

\begin{proof}
To prove Lemma \ref{lem7}, we find an adequate upperbound on the functions $\alpha$ and $\beta$ used in  Lemma \ref{lem6new}.   For this,  we use the following inequalities:
\begin{align*}
u_2(t,t_o)^2 & \leq P^2 \alpha_2^{\frac{2}{3}} V_2(t_o)^2, 
\quad
|g(t,t_o)|  \leq  e^{2 \delta_1 (t-t_o)} P |\alpha_2^{\frac{1}{3}} V_2(t_o)|,
\\ 
2 C \int^{t}_{t_o} u_2(\tau) g(t,\tau) d\tau & =  2 C \int^{t}_{t_o} u_2(\tau)^2 e^{2 \delta_1 (t-\tau)}  d\tau 
 \leq  2 C e^{2 \delta_1 (t-t_o)} \int^{t}_{t_o} u_2(\tau)^2  d\tau 
\\ & \leq  2 C e^{2 \delta_1 (t-t_o)} (t - t_o) P^2  \alpha_2^{\frac{2}{3}} V_2(t_o)^2, 
\end{align*}
and 
$$ \bigg| \frac{d}{d\tau} g(t,\tau) \bigg|  \leq 
 2 \delta_2 e^{2\delta_1(t-\tau)} |u_2(\tau)|   \leq 2 \delta_2 e^{2\delta_1(t-\tau)}P \alpha_2^{\frac{1}{3}} V_2(t_o). $$ 

As a result, we have 
\begin{align*}
 \int^{t}_{t_o} \bigg| \frac{d}{d\tau} g(t,\tau) \bigg|   d\tau   \leq  e^{2\delta_1(t-t_o)} 2 (t - t_o) \delta_2 P  \alpha_2^{\frac{1}{3}} V_2(t_o).
\end{align*}

As a consequence, we derive the following upperbounds on $\alpha$ and $\beta$. 
\begin{align*}
\alpha(t,t_o) & \leq  (2 + P) e^{2\delta_1 (t-t_o)} V_1(t_o)  
 +
2 C e^{2 \delta_1 (t-t_o)} P^2 (t-t_o) 
\alpha_2^{\frac{2}{3}} V_2(t_o)^2   \\ & +
2 P^2  \alpha_2^{\frac{2}{3}} V_2(t_o)^2  +  e^{2 \delta_1 (t-t_o)} P   \alpha_2^{\frac{2}{3}} V_2(t_o)^2
\\ & + P \delta_2  \exp^{2\delta_1(t-t_o)} (t-t_o) \alpha_2^{\frac{1}{3}} V_2(t_o).
\end{align*}
Next, we introduce the following functions:
\begin{align*}
g_1(t-t_o) & :=  (2 + P) e^{2\delta_1 (t-t_o)}, 
\\
g_2\left( t-t_o \right) & :=
2 C e^{2 \delta_1 (t-t_o)} P^2  (t - t_o)  +
2 P^2  +  e^{2 \delta_1 (t-t_o)} P,
\\ 
g_3(t-t_o) & := P \delta_2  \exp^{2\delta_1(t-t_o)} (t-t_o).
\end{align*}
Hence, the upperbound on $\alpha$ can be expressed as follows 
\begin{align*}
\alpha(t,t_o) & \leq  g_1 (t-t_o) V_1(t_o) + g_2(t-t_o)  \alpha_2^{\frac{2}{3}} V_2(t_o)^2   + g_3(t-t_o)  \alpha_2^{\frac{1}{3}} V_2(t_o).
\end{align*}

Now, to propose an adequate upperbound  on $\beta$, we note that 
\begin{align*}
\beta(t,\tau,t_o) := 2  \bigg|
\frac{d}{d\tau} g(t,\tau,t_o) \bigg|   \leq  4 \delta_2 P e^{2\delta_1(t-\tau)}   \alpha_2^{\frac{1}{3}} V_2(t_o).
\end{align*}
Next, we let the functions 
\begin{align*}
h_1\left( t- \tau \right)  :=~  & 4 \delta_2 P e^{2\delta_1(t-\tau)}.
\end{align*} 
Hence, the upperbound on $\beta$ can be expressed as follows:
\begin{align*}
\beta(t,\tau,t_o)  \leq   h_1 \left(t- \tau \right) \alpha_2^{\frac{1}{3}} V_2(t_o). 
\end{align*}

Finally, to  conclude the proposed upperbound on $V_1$, we choose the functions $M_1$ and $M_2$, independent of $a$, as follows:
\begin{align*}
M_1(T) & := \sup  \left\{g_{2,3} \left( \tau \right) :  \tau \in [0,T]  \right\},
\\
M_2(T) & := \sup  \left\{  h_1\left(\tau \right) : \tau \in [0,T] \right\}.
\end{align*}
As a result, for all $t \in [t_o, t_o + T]$, we have 
\begin{align*}
 \alpha(t,t_o) & \leq   g_1(T) V_1(t_o) + M_1(T) \left[  \alpha_2^{\frac{2}{3}} V_2(t_o)^2  + \alpha_2^{\frac{1}{3}} V_2(t_o) \right] =  \bar{\alpha}(t_o,T)
\end{align*} 
and
\begin{align*}
\beta(t,\tau,t_o)  \leq  M_2(T) \left( \alpha_2^{\frac{1}{3}} V_2(t_o)  \right)   =  \beta_1(t_o,T). 
\end{align*}
As a result, an upperbound  on $V_1(t)$ when $t \in [t_o,t_o + T]$ is given by
\begin{align*}
 V_1(t) &  \leq  \alpha(t,t_o) + \int^{t}_{t_o} \alpha(\tau,t_o) \beta(t,\tau,t_o) 
e^{ \left[\int^{t}_{\tau}  \beta(t,s,t_o) ds \right] } d\tau  
 \\ & \leq  \bar{\alpha}(t_o,T) + T
\bar{\alpha}(t_o,T)  \beta_1(t_o,T)  e^{  \beta_1(t_o,T) T }.
\end{align*}
\end{proof}

\begin{remark} \label{rem6}
According to  \eqref{eqV1upbd}, we conclude that 
$$V_1(t_o)  \leq \lim_{T \rightarrow 0}  \mathcal{V}_1 \left(T, V_1(t_o), \alpha_2^{\frac{1}{3}} V_2(t_o) \right).  $$
Note that 
\begin{align*}
\lim_{T \rightarrow 0}   \mathcal{V}_1 \left(T, \cdot \right) & =
\lim_{T \rightarrow 0}  \bar{\alpha}(t_o, T)  
\\ & = (2+P)  V_1(t_o) 
 + \left( \lim_{T \rightarrow 0}  M_1(T) \right)  
\left[ \alpha_2^{\frac{2}{3}} V_2(t_o)^2  + \alpha_2^{\frac{1}{3}} V_2(t_o) \right].
\end{align*} 
Next,  from the proof of Lemma \ref{lem7},  we note that  $\lim_{T \rightarrow 0}  M_1(T) = 4 P^2 + P$. As a result, we obtain
\begin{align*}
 \lim_{T \rightarrow 0}    \mathcal{V}_1 \left(T, \cdot \right)  & = (2+P)  V_1(t_o)  + \left( 4P^2 + P \right)  \left[ \alpha_2^{\frac{2}{3}} V_2(t_o)^2  +  \alpha_2^{\frac{1}{3}} V_2(t_o)  \right].
\end{align*} 
\end{remark} 

\fi

\subsection{Proof of Theorem \ref{thm2}} 
To prove Theorem \ref{thm2},  we show that
 $\Sigma_4$ satisfies the following properties:
 
\begin{enumerate} [label={SGB.},leftmargin=*]
\item \label{item:SGB} For each $\beta>0$, there exists $\gamma >0 $ and  $(\alpha_1,\alpha_2) \in \mathbb{R}_{>0} \times \mathbb{R}_{>0}$ sufficiently large such that, for every solution $W$ with $W(t_o) \leq \beta$,  $W(t_o + t) \leq  \gamma$ for all $t \geq  0$.  
\end{enumerate}

\begin{enumerate} [label={PSGA.},leftmargin=*]
\item \label{item:PSGAt} For each $\beta > \epsilon >0$, there exist $(\alpha_1,\alpha_2) \in \mathbb{R}_{>0} \times 
\mathbb{R}_{>0}$ sufficiently large such that, every solution $W := (V_1,V_2)$ to 
$\Sigma_4$ with $W(t_o) \leq \beta$, there exists $T>0$ such that 
$W(t_o + T) \leq \epsilon$.
\end{enumerate}

\begin{enumerate} [label={SGUB.},leftmargin=*]
\item \label{item:SGUB} There exists $\gamma>0$ such that, for each $\beta>0$, there exists $(\alpha_1,\alpha_2) \in \mathbb{R}_{>0} \times \mathbb{R}_{>0}$ sufficiently large such that, for every solution $W$ with $W(t_o) \leq \beta$, there exists $T>0$ such that 
$W(t_o + t) \leq  \gamma$ for all $t \geq  T$.  
\end{enumerate}

\underline{Proof of  \ref{item:SGB}}
Consider the sequence $\{ t_k \}^{\infty}_{k=1}$ used in the proposed sensing scenario.  Note that,  for all 
$t \in [t_{2k+1}, t_{2k+2}] \subset I_1$,  we have 
\begin{align*}
V_1(t)  \leq  e^{- \alpha_1 (t - t_{2k+1})} V_1(t_{2k+1}),   
 \quad 
V_2(t) =   e^{2 \delta_1 (t - t_{2k+1})} V_2(t_{2k+1}).    
 \end{align*}
Furthermore, for each $t \in [t_{2k}, t_{2k+1}] \subset I_2$,  we use Lemma \ref{lem7} to conclude that
 \begin{align*}
V_2(t) & \leq  e^{-  \alpha_2^{\frac{1}{3}} (t - t_{2k})} V_2(t_{2k}) + b_o \alpha_2^{-\frac{1}{3}}
\left[ 1- e^{- \alpha_2^{\frac{1}{3}}(t-t_{2k})} \right].    
\\
V_1(t) & \leq  \mathcal{V}_1 \left(\bar{T}_2, V_1(t_{2k}), \alpha_2^{\frac{1}{3}} V_2(t_{2k}) \right). 
\end{align*} 
Now, we let $W := V_1 + V_2$ and we note that
\begin{align*} 
 W(t_{2k+2}) &  \leq  
e^{ - \alpha_1 \bar{T}_1} 
\mathcal{V}_1\left(\bar{T}_2, V_1(t_{2k}), \alpha_2^{\frac{1}{3}} V_2(t_{2k}) \right) 
\\ &  +
e^{2 \delta_1 \bar{T}_1 - \alpha_2^{\frac{1}{3}} \bar{T}_2} V_2(t_{2k})  +  b_o e^{2 \delta_1 \bar{T}_1}  \alpha_2^{-\frac{1}{3}} 
\left[ 1- e^{- \alpha_2^{\frac{1}{3}} \bar{T}_2} \right] .   
\end{align*}
Now, we choose $\alpha^{\star}_2$ large enough so that 
$$  2\delta_1 \bar{T}_1 -  \alpha_2^{\star \frac{1}{3}} \bar{T}_2  = -2. $$
Hence,  for each $\alpha_2 \geq \alpha^{\star}_2$, we have 
\begin{equation*}
\begin{aligned} 
  W(t_{2k+2})  &  \leq 
  e^{ - \alpha_1 \bar{T}_1} 
\mathcal{V}_1\left(\bar{T}_2, V_1(t_{2k}), \alpha_2^{\frac{1}{3}} V_2(t_{2k}) \right) 
 \\ & ~
+ e^{- 2} W(t_{2k})  +  b_o e^{2 \delta_1 \bar{T}_1} 
\left[ 1- e^{- \alpha_2^{\frac{1}{3}} \bar{T}_2} \right] \alpha_2^{-\frac{1}{3}}.   
\end{aligned}
\end{equation*}
Now,  since $\mathcal{V}_1$ is independent  on $\alpha_1$, continuous, and non-decreasing  in each of its arguments, we conclude that, for each $\beta > 0$ such that $W(t_{2k})  \leq \beta$ and for each $\alpha_2 \geq \alpha_2^{\star}$, we can find $\alpha^*_1$ such that
\begin{align*}
 \mathcal{V}_1\left(\bar{T}_2, V_1(t_{2k}), \alpha_2^{\frac{1}{3}} V_2(t_{2k}) \right)  \leq e^{\frac{\alpha_1 \bar{T}_1}{2}}   W(t_{2k}).   
\end{align*}
Hence,  for each $\alpha_1 \geq \alpha^*_1$, we have 
\begin{equation}
\label{eqboundperlin}
\begin{aligned}
W(t_{2k+2})  & \leq   \left( e^{- \frac{\alpha_1 \bar{T}_1}{2}}  + 
e^{- 2} \right)  W(t_{2k})   +  e^{2 \delta_1 \bar{T}_1}  b_o  \left[ 1- e^{- \alpha_2^{\frac{1}{3}} \bar{T}_2} \right] \alpha_2^{-\frac{1}{3}}.
 \end{aligned}
 \end{equation}
As a result,  when $(\alpha_1,\alpha_2)$ are sufficiently large, we conclude that 
\begin{align} \label{eqboundper}
W(t_{2k})   \leq   \beta   \qquad  \forall k \in \{1,2,...  \}.
 \end{align} 
Under \eqref{eqboundper},  to complete the proof of \ref{item:SGB}, it is enough to show the existence of 
$\beta_1>0$ such that 
\begin{align*}
W(t)   \leq   \beta_1   \qquad \forall t \in [t_o, t_{2}]. 
 \end{align*} 
 
To this end, we first note that, for each $t \in [t_1, t_2]$, we have
\begin{align*}
V_1(t)   \leq   e^{- \alpha_1 (t -T_1 - \bar{T}_2)} V_1(t_1),
\quad 
V_2(t)  =  e^{2 \delta_1 (t - T_1 - \bar{T}_2)} V_2(t_1).    
 \end{align*}
Hence, for each $t \in [t_1 , t_2]$,  we have 
 \begin{align*}
W(t) \leq  e^{2 \delta_1 \bar{T}_1} W(t_1)
 \end{align*}
Furthermore, for each $t \in [t_o, t_1]$,  we use Lemma \ref{lem7} to conclude that
 \begin{align*}
V_2(t)  \leq  e^{- \alpha_2^{\frac{1}{3}} (t-t_o)} V_2(t_o) + b_o \alpha_2^{-\frac{1}{3}}, 
\quad
 V_1(t) &  \leq  \mathcal{V}_1 \left( \bar{T}_2,  V_1(t_o), \alpha_2^{\frac{1}{3}} V_2(t_o) \right). 
\end{align*} 
Now, since $\mathcal{V}_1$ is non-decreasing in its arguments, we conclude that,  or each $t \in [t_o, t_1]$,  we have 
 \begin{align*}
V_2(t)  \leq   \beta + b_o \alpha_2^{-\frac{1}{3}},   
\quad
 V_1(t)   \leq  \mathcal{V}_1 \left( \bar{T}_2,  \beta, \alpha_2^{\frac{1}{3}} \beta \right). 
\end{align*} 
As a result,  for each $t \in [t_o, t_2]$, we conclude that
\begin{align*} 
 W(t)  \leq \beta_1 :=  \left[  \beta + b_o \alpha_2^{-\frac{1}{3}}
+ \mathcal{V}_1 \left( \bar{T}_2,  \beta, \alpha_2^{\frac{1}{3}} \beta \right)\right] e^{2 \delta_1 \bar{T}_1}.
\end{align*}

\underline{Proof of  \ref{item:PSGAt}}  Consider the sequence 
$\{T_i\}^{\infty}_{i = 1}$ with  $T_i = t_{2i}$ for all 
$i \in \{1,2,... \}$.  Note that 
$$  T_{i+1} - T_i = \bar{T}_1 + \bar{T_2} \qquad \forall i \in \{0,1,2,...\}.  $$
Also,  we let 
\begin{align*}
q   :=   e^{- \alpha_1 \bar{T}_1/2}  + e^{- 2}, ~~
p  :=  e^{2 \delta_1 \bar{T}_1}  b_o
\left[ 1- e^{- \alpha_2^{\frac{1}{3}}  \bar{T}_2} \right] \alpha_2^{-\frac{1}{3}}.
\end{align*}
Note that $p$ can be made arbitrarily small by choosing $\alpha_2$ and $\alpha_1$ sufficiently large. 
Hence, without loss of generality, we let
$q  :=   2 e^{- 2}$ and $p  :=  e^{2 \delta_1 \bar{T}_1} b_o \alpha_2^{-\frac{1}{3}}$.
As a result,   using \eqref{eqboundperlin}, we obtain
\begin{align*}
W(T_{i+1})  \leq &~  q  W(T_i) + p \qquad \forall i \in \{1,2,... \}.
\end{align*}
Hence, for each $i \in \{1,2,... \}$,
\begin{align*}
W(T_{i+1})  \leq  q^{i-1}  W(T_1) + \sum^{i-2}_{k=0} q^k p  \leq   q^{i-1}  W(T_1) + p \frac{1- q^{i-1}}{1-q}.
\end{align*}
As a result, we have  
$$ \lim_{i \rightarrow \infty} W(T_{i})  \leq  \frac{p}{1-q} =  \frac{b_o e^{2 \delta_1 \bar{T}_1}}{1-2 e^{-2}} \alpha_2^{-\frac{1}{3}} .  $$ 

As a consequence, we conclude the existence of $M>1$ and $i^* \in \{1,2,... \}$ such that 
\begin{align} \label{eqtouse} 
 W(T_i)  \leq  \frac{e^{2 \delta_1 \bar{T}_1}}{1-2 e^{-2}} b_o \alpha_2^{-\frac{1}{3}} =: M b_o  \alpha_2^{-\frac{1}{3}}    \qquad \forall i \geq i^*.   
 \end{align}
 
\underline{Proof of \ref{item:SGUB}} 
Consider $\beta>0$ such that $W(t_o) \leq \beta$.  
According to the proof of  \ref{item:PSGAt},  when $(\alpha_1,\alpha_2)$ are appropriately chosen, there exists $i^* \geq \{1,2,... \}$ such that, for each $i \geq i^*$, we have  
$$ V_1(T_{i}) + V_2(T_{i}) = W(T_{i})  \leq M b_o \alpha_2^{-\frac{1}{3}}. $$ 
Hence, it is enough to show the existence of $\gamma>0$ such that,  when $W(T_1) \leq M b_o \alpha_2^{-\frac{1}{3}}$,  
it follows that 
$$   W(t) \leq \gamma  \qquad  \forall t \in [T_1,  T_2].  $$

To this end, we first note that, for each $t \in [T_1 + \bar{T}_2, T_2]$, we have
\begin{align*}
V_1(t)  \leq &~    e^{- \alpha_1 (t -T_1 - \bar{T}_2)} V_1(T_1 + \bar{T}_2)  
 \\
V_2(t) =  &~   e^{2 \delta_1 (t - T_1 - \bar{T}_2)} V_2(T_1 + \bar{T}_2).    
 \end{align*}
Hence, for each $t \in [T_1 + \bar{T}_2, T_2]$,  we have 
 \begin{align*}
W(t) \leq  e^{2 \delta_1 \bar{T}_1} W(T_1 + \bar{T}_2).
 \end{align*}
Furthermore, for each $t \in [T_1, T_1 + \bar{T}_2]$,  we use Lemma \ref{lem7} to conclude that
 \begin{align*}
V_2(t) & \leq  e^{- \alpha_2^{\frac{1}{3}} (t-T_1)} V_2(T_1) + b_o \alpha_2^{-\frac{1}{3}}.    
\\
 V_1(t) &  \leq  \mathcal{V}_1 \left(\bar{T}_2,  V_1(T_1), \alpha_2^{\frac{1}{3}} V_2(T_1)  \right). 
\end{align*} 

Now, since $\mathcal{V}_1$ is non-decreasing in its arguments, we conclude that, for $\alpha_2 \geq 1$ and for each $t \in [T_1, T_1 + \bar{T}_2]$, we have 
 \begin{align*}
V_2(t) \leq  ~ &  (M + 1) b_o  \alpha_2^{-\frac{1}{3}} .    
\\
 V_1(t)  \leq ~ & \mathcal{V}_1 \left(\bar{T}_2,  \alpha_2^{-\frac{1}{3}} Mb_o ,  Mb_o \right). 
\end{align*} 
As  result, for each $t \in [T_1, T_2]$, we have
\begin{align*}  
  W(t)   \leq \gamma & :=  e^{2 \delta_1 \bar{T}_1}  \mathcal{V}_1 \left(\bar{T}_2,  \alpha_2^{-\frac{1}{3}} Mb_o, Mb_o \right) 
 +   e^{2 \delta_1 \bar{T}_1}  (M + 1)b_o \alpha_2^{-\frac{1}{3}}.
\end{align*}
\hfill $\blacksquare$

\begin{remark}  
As studied in \cite{liu2001stability},  when $\lambda_1 < 4 \pi^2$,  the  $\mathcal{L}^2$ global exponential stabilization of the trivial solution to $\Sigma_2$  (thus to $\Sigma_1$) is a straightforward task.  Indeed,  using Lemma \ref{lem3},  we conclude that the solution  $\delta_o$ to Problem \ref{prob1}; while replacing $\lambda$ therein by the parameter  $\lambda_1$ of $\Sigma_2$,  is positive.  Hence,  by taking  $u_1 = u_2 = u_3  = 0$,  the resulting solution $(w,v)$ to $\Sigma_2$ is in $H^2_{o}(0,Y) \times H^2_{o}(Y,L)$. Hence, using  Lemma \ref{lem4-}, we conclude that $\dot{V}_1  \leq   - \delta_o \int^{Y}_{0} w(x)^2 dx = - 2\delta_o V_1$ and $\dot{V}_2 \leq - \delta_o \int^{L}_{Y} v(x)^2 dx = - 2\delta_o V_2$,  which implies $\mathcal{L}^2$-GES of the trivial solution.
\end{remark}

\begin{remark}
The restrictive nature of the upper bound derived in Lemma \ref{lem7} does not allow us to verify $\mathcal{L}^2$-PS for $\Sigma_1$ by analyzing $\Sigma_4$.  
Indeed,  according to the proof of $\mathcal{L}^2$-PSGA, 
we are able to show, when $k$ is large, that,  for some $M>0$,
$ ||u(t_{2k})||  \leq   M b_o \alpha_2^{-\frac{1}{3}}$.  
However,  we are not able to find a class $\mathcal{K}$ function $\kappa$ such that,  for each  $t \in [t_{2k}, t_{2k+2}]$,  
$||u(t)||  \leq   \kappa \left(||u(t_{2k})|| \right) + M b_o \alpha_2^{-\frac{1}{3}}$.
\end{remark}

\section{Conclusion and Future Work}

This paper proposed two boundary controllers to stabilize the origin of the nonlinear Kuramoto-Sivashinsky equation,  under intermittent measurements.  Using the first controller,  we are able to provide stronger stability properties compared with the second one.  In future work,  we would like to improve the stability properties of the second controller and consider the case where the coefficient $\lambda_1$ is unknown.  
 Furthermore,   boundary control of $\Sigma_1$ when $\lambda_1 \geq 4 \pi^2$ and while measuring $u, u_x,u_{xx},....$ at isolated points instead of intervals, is an open question.   A potential  way to generalize the existing results requiring measurements on intervals consists in designing an observer capable of reconstructing $u$ on intervals starting from  isolated-points measurements.

\section{Acknowledgment}

The authors are thankful to Camil Belhadjoudja for pointing out an issue in the proof of Lemma \ref{lem7} and suggesting a way to fix it.

\bibliography{biblio}

\def\loria{Loria} \def\nesic{Ne\v{s}i\'{c}\,}\def\nonumero{\def\numerodeitem{}}
\begin{thebibliography}{10}

\bibitem{kuramoto1980instability}
Y.~Kuramoto, ``Instability and turbulence of wavefronts in reaction-diffusion
  systems,'' {\em Progress of Theoretical Physics}, vol.~63, no.~6,
  pp.~1885--1903, 1980.

\bibitem{hac1993sensor}
A.~Ha{\'c} and L.~Liu, ``Sensor and actuator location in motion control of
  flexible structures,'' {\em Journal of sound and vibration}, vol.~167, no.~2,
  pp.~239--261, 1993.

\bibitem{rouchon2008quantum}
P.~Rouchon, ``Quantum systems and control 1,'' {\em Revue Africaine de la
  Recherche en Informatique et Math{\'e}matiques Appliqu{\'e}es}, vol.~9, 2008.

\bibitem{schneider2017nonlinear}
G.~Schneider and H.~Uecker, {\em Nonlinear PDEs}, vol.~182.
\newblock American Mathematical Soc., 2017.

\bibitem{karafyllis2019input}
I.~Karafyllis and M.~Krstic, {\em Input-to-state stability for PDEs}.
\newblock Springer, 2019.

\bibitem{kocarev1997synchronizing}
L.~Kocarev, Z.~Tasev, and U.~Parlitz, ``Synchronizing spatiotemporal chaos of
  partial differential equations,'' {\em Physical Review Letters}, vol.~79,
  no.~1, p.~51, 1997.

\bibitem{guo2020robust}
B.-Z. Guo and T.~Meng, ``Robust error based non-collocated output tracking
  control for a heat equation,'' {\em Automatica}, vol.~114, p.~108818, 2020.

\bibitem{armaou1999nonlinear}
A.~Armaou and P.~Christofides, ``Nonlinear feedback control of parabolic
  partial differential equation systems with time-dependent spatial domains,''
  {\em Journal of mathematical analysis and applications}, vol.~239, no.~1,
  pp.~124--157, 1999.

\bibitem{kang2018distributed}
W.~Kang and E.~Fridman, ``{Distributed sampled-data control of
  Kuramoto--Sivashinsky equation},'' {\em Automatica}, vol.~95, pp.~514--524,
  2018.

\bibitem{tasev2000synchronization}
Z.~Tasev, L.~Kocarev, L.~Junge, and U.~Parlitz, ``{Synchronization of
  Kuramoto--Sivashinsky equations using spatially local coupling},'' {\em
  International Journal of Bifurcation and Chaos}, vol.~10, no.~04,
  pp.~869--873, 2000.

\bibitem{khadra2005impulsive}
A.~Khadra, X.~Liu, and X.~Shen, ``Impulsive control and synchronization of
  spatiotemporal chaos,'' {\em Chaos, Solitons \& Fractals}, vol.~26, no.~2,
  pp.~615--636, 2005.

\bibitem{krstic2008boundary}
M.~Krstic and A.~Smyshlyaev, {\em {Boundary control of PDEs: A course on
  backstepping designs}}.
\newblock SIAM, 2008.

\bibitem{krstic2008adaptive}
M.~Krstic and A.~Smyshlyaev, ``{Adaptive boundary control for unstable
  parabolic PDEs—Part I: Lyapunov design},'' {\em IEEE Transactions on
  Automatic Control}, vol.~53, no.~7, pp.~1575--1591, 2008.

\bibitem{1102875}
R.~{Curtain}, ``Finite-dimensional compensator design for parabolic distributed
  systems with point sensors and boundary input,'' {\em IEEE Transactions on
  Automatic Control}, vol.~27, no.~1, pp.~98--104, 1982.

\bibitem{laila20063}
D.~Laila, D.~Ne{\v{s}}i{\'c}, and A.~Astolfi, ``Sampled-data control of
  nonlinear systems,'' in {\em Advanced topics in control systems theory},
  pp.~91--137, Springer, 2006.

\bibitem{yang2001impulsive}
T.~Yang, {\em Impulsive control theory}, vol.~272.
\newblock Springer Science \& Business Media, 2001.

\bibitem{parmananda1997generalized}
P.~Parmananda, ``Generalized synchronization of spatiotemporal chemical
  chaos,'' {\em Physical Review E}, vol.~56, no.~2, p.~1595, 1997.

\bibitem{junge2000synchronization}
L.~Junge and U.~Parlitz, ``Synchronization and control of coupled
  ginzburg-landau equations using local coupling,'' {\em Physical Review E},
  vol.~61, no.~4, p.~3736, 2000.

\bibitem{katz2020finite}
R.~Katz and E.~Fridman, ``{Finite-dimensional control of the
  Kuramoto-Sivashinsky equation under point measurement and actuation},'' in
  {\em 2020 59th IEEE Conference on Decision and Control (CDC)},
  pp.~4423--4428, IEEE, 2020.

\bibitem{Toshidaref}
K.~Toshihiro, ``{Adaptive stabilization of the Kuramoto-Sivashinsky
  equation},'' {\em International Journal of Systems Science}, vol.~33, no.~3,
  pp.~175--180, 2002.

\bibitem{liu2001stability}
W.-J. Liu and M.~Krsti{\'c}, ``{Stability enhancement by boundary control in
  the Kuramoto--Sivashinsky equation},'' {\em Nonlinear Analysis: Theory,
  Methods \& Applications}, vol.~43, no.~4, pp.~485--507, 2001.

\bibitem{sakthivel2007non}
R.~Sakthivel and H.~Ito, ``{Non-linear robust boundary control of the
  Kuramoto--Sivashinsky equation},'' {\em IMA Journal of Mathematical Control
  and Information}, vol.~24, no.~1, pp.~47--55, 2007.

\bibitem{guzman2019stabilization}
P.~Guzm{\'a}n, S.~Marx, and E.~Cerpa, ``{Stabilization of the linear
  Kuramoto-Sivashinsky equation with a delayed boundary control},'' {\em
  IFAC-PapersOnLine}, vol.~52, no.~2, pp.~70--75, 2019.

\bibitem{coron2015fredholm}
J.-M. Coron and Q.~L{\"u}, ``{Fredholm transform and local rapid stabilization
  for a Kuramoto--Sivashinsky equation},'' {\em Journal of Differential
  Equations}, vol.~259, no.~8, pp.~3683--3729, 2015.

\bibitem{zhang2001stability}
W.~Zhang, M.~S. Branicky, and S.~M. Phillips, ``Stability of networked control
  systems,'' {\em IEEE control systems magazine}, vol.~21, no.~1, pp.~84--99,
  2001.

\bibitem{sun2009networked}
Y.~Sun, S.~Ghantasala, and N.~H. El-Farra, ``Networked control of spatially
  distributed processes with sensor-controller communication constraints,'' in
  {\em 2009 American Control Conference}, pp.~2489--2494, IEEE, 2009.

\bibitem{zhang2019fuzzy}
X.-W. Zhang and H.-N. Wu, ``{Fuzzy stabilization design for semilinear
  parabolic PDE systems with mobile actuators and sensors},'' {\em IEEE
  Transactions on Fuzzy Systems}, vol.~28, no.~3, pp.~474--486, 2019.

\bibitem{khapalov1992observability}
A.~Y. Khapalov, ``Observability of parabolic systems with scanning sensors,''
  in {\em [1992] Proceedings of the 31st IEEE Conference on Decision and
  Control}, pp.~1311--1312, IEEE, 1992.

\bibitem{mu2014improving}
W.~Mu, B.~Cui, W.~Li, and Z.~Jiang, ``Improving control and estimation for
  distributed parameter systems utilizing mobile actuator--sensor network,''
  {\em ISA transactions}, vol.~53, no.~4, pp.~1087--1095, 2014.

\bibitem{5406054}
M.~A. Demetriou, ``Guidance of mobile actuator-plus-sensor networks for
  improved control and estimation of distributed parameter systems,'' {\em IEEE
  Transactions on Automatic Control}, vol.~55, no.~7, pp.~1570--1584, 2010.

\bibitem{ZHANG20201299}
X.-W. Zhang and H.-N. Wu, ``{Switching state observer design for semilinear
  parabolic PDE systems with mobile sensors},'' {\em Journal of the Franklin
  Institute}, vol.~357, no.~2, pp.~1299--1317, 2020.

\bibitem{teel1999semi}
A.~A.~R.~Teel, J.~Peuteman, and D.~Aeyels, ``Semi-global practical asymptotic
  stability and averaging,'' {\em Systems \& control letters}, vol.~37, no.~5,
  pp.~329--334, 1999.

\bibitem{ames1997inequalities}
B.~G. Pachpatte, {\em Inequalities for differential and integral equations},
  vol.~197.
\newblock Mathematics in Science and Engineering, 1997.

\end{thebibliography}
\bibliographystyle{ieeetr}     

\ifitsdraft

\appendix

\subsection{Intermediate Results}
We start recalling the following classical inequality.  
\begin{lemma} \label{lem5}
For all $(a,b) \in \mathbb{R} \times \mathbb{R}$, and for all
$\epsilon > 0$, we have
$$ a b \leq \frac{1}{2 \epsilon} a^2 + \frac{\epsilon}{2} b^2. $$
\end{lemma} 
\begin{proof}
To establish the proof, it is enough to note that
$$  \left( \frac{a}{\sqrt{2 \epsilon}}  - \sqrt{\frac{\epsilon}{2}} b   \right)^2 \geq 0. $$
\end{proof}

Next, we recall from \cite{ames1997inequalities} the classical Gronwall-Bellman inequality.

\begin{lemma} \label{lem5-}
Given $T>0$ and two bounded, nonnegative, and measurable functions $V : [0,T] \rightarrow \mathbb{R}$ and $\alpha : [0,T] \rightarrow \mathbb{R}$, given nonnegative integrable function
$\beta:[0,T] \rightarrow \mathbb{R}$, we assume that 
$$ V(t) \leq \alpha(t) + \int^{t}_{0} \beta(s) V(s) ds \qquad \forall t \in [0,T]. $$
Then, for each $t \in [0,T]$, we have 
$$ V(t) \leq \alpha(t) + 
\int^{t}_{0} \alpha(s) \beta(s) 
e^{\int^{t}_{s} \beta(\tau) d \tau } ds. $$
\end{lemma} 
 
In the following lemma,  we show how to numerically compute $\delta_o$ the solution to  Problem \ref{prob1}.  Furthermore, we explicitly compute a 
lower-bound of $\delta_o$. 

\begin{lemma} \label{lemlowbound}
Given $\lambda \geq 4 \alpha_2^{\frac{1}{3}} V_2^2$, we let $\delta_o$ be the corresponding solution Problem \ref{prob1} with $(a,b) := (0, 1)$.  
Hence,  $\delta_o$ can be computed as
\begin{equation*}
\begin{aligned}
\delta_o &  := \min \{ \delta_{o1},\delta_{o2} \},
\\
 \delta_{o1} & := \min \{ \delta \in [- \lambda^2/4,0] : \det \mathcal{A}_1 = 0  \},  
\\ 
 \delta_{o2} & := \min \{ \delta \leq [ \delta^\star_{o2} , - \lambda^2/4] : \det \mathcal{A}_2 = 0  \},  
\end{aligned}
\end{equation*}
where  $4 \delta^\star_{o2}  :=  - \left[ \left( \frac{\lambda} {2 \ln \left( 13/12 \right)}  \right)^2 + \lambda + 4 x_o^2 \right]^{2}$,  
$x_o>0$ is such that,  for each 
$x \geq x_o$,  we have  
\begin{align*}
-  \exp^{2 x} x^2  + 12 \ln(2)    \exp^{2x} x  & +  
 80 \ln(2)  x  + 56  x^2 + 12 \ln(2)^2 \leq 0,
 \end{align*}
and the matrices $\mathcal{A}_1$ and $\mathcal{A}_2$ are given in \eqref{eqA1mat} and \eqref{eqA2mat}, respectively. 
\end{lemma}

\begin{proof}
We start noting that the characteristic polynomial associated to \eqref{eigenprob} is given by $P(x) := x^4  + \lambda x^2 - \delta$, where  
$\delta \leq 0$ and $\lambda \geq 4 \alpha_2^{\frac{1}{3}} V_2^2$.  
The polynomial $P$ has the following roots
\begin{align*}
x_1 & := \frac{\sqrt{ - \lambda + \sqrt{\lambda^2 + 4 \delta}}}{\sqrt{2}}, ~ 
x_2 := \frac{\sqrt{- \lambda - \sqrt{\lambda^2 + 4 \delta}  }}{\sqrt{2}} \\
x_3 & := - \frac{\sqrt{ - \lambda + \sqrt{\lambda^2 + 4 \delta}}}{\sqrt{2}}, ~
x_4 :=  - \frac{ \sqrt{ - \lambda - \sqrt{\lambda^2 + 4 \delta}}}{\sqrt{2}}.
\end{align*}

At this point, we distinguish between the following two situations:

\begin{enumerate}
\item When $\delta \in [-\frac{\lambda^2}{4}, 0]$, the roots are complex conjugate and located on the imaginary axis.  In particular, we have  
\begin{align*}
x_1 & := i \frac{\sqrt{  \lambda - \sqrt{\lambda^2 + 4 \delta}}}{\sqrt{2}} =: i \omega_1  \\ 
x_2 & := i \frac{\sqrt{ \lambda +  \sqrt{\lambda^2 + 4 \delta}  }}{\sqrt{2}}  =: i \omega_2  \\
x_3 & := - i \frac{\sqrt{  \lambda - \sqrt{\lambda^2 + 4 \delta}}}{\sqrt{2}}  =: - i \omega_1  \\ 
x_4 & :=  -i \frac{ \sqrt{ \lambda + \sqrt{\lambda^2 + 4 \delta}}}{\sqrt{2}} =: - i \omega_2.
\end{align*}

Hence,  the solutions to the differential equation in  \eqref{eigenprob} are of the form 
\begin{align*}
 z(x) := C_1 \sin (\omega_1 x)  + C_2 \cos (\omega_1 x) 
 + C_3 \sin (\omega_2 x) + C_4 \cos (\omega_2 x),   
 \end{align*}
where $(C_1,C_2, C_3,C_4)$ are free parameters subject to the boundary conditions in \eqref{eqBC}. Hence,  the solution $z$ is nontrivial if and only if 
$$ \mathcal{A}_1 \begin{bmatrix}  C_1 & C_2 & C_3 & C_4  \end{bmatrix}^\top  = 0,  $$
 where 
\begin{align} \label{eqA1mat}
\hspace{-0.6cm} \mathcal{A}_1 := 
\begin{bmatrix}       
0 & 1 & 0 & 1  \\
\sin (\omega_1) &  \cos (\omega_1) &  \sin (\omega_2) &  \cos (\omega_2)  \\
\omega_1 & 0 & \omega_2 & 0 \\
\omega_1 \cos (\omega_1) & - \omega_1  \sin (\omega_1) &  \omega_2 
 \cos (\omega_2) & - \omega_2  \sin (\omega_2)
\end{bmatrix}
\end{align} 
\item When $\delta < - \lambda^2/4$,  the solutions are complex conjugate. That is,  by letting $\omega_3  := \sqrt{- \lambda^2 - 4 \delta}$, we obtain 
\begin{align*}
x_1 & :=  \frac{\sqrt{ - \lambda + i \omega_3}}{\sqrt{2}},  \quad  
x_2  := \frac{\sqrt{- \lambda - i \omega_3}}{\sqrt{2}},  \\
x_3 & := -  \frac{\sqrt{ -  \lambda + i \omega_3}}{\sqrt{2}},   \quad 
x_4  :=  - \frac{ \sqrt{-  \lambda - i \omega_3}}{\sqrt{2}}.
\end{align*}
Now, we let $r := \sqrt{\frac{\sqrt{\lambda^2 + \omega_3^2}}{2}}$ and $\theta := \frac{\arctan (\omega_3/\lambda)}{2}$ to conclude that 
$x_1 := r \exp^{i \theta}$,  $x_2  :=  r \exp^{- i \theta}$,  
$x_3 := -  r \exp^{i \theta}$,   $x_4  :=  -  r \exp^{- i \theta}$.
Finally,   we let $\omega_4 := r \cos(\theta)  $ and $\omega_5 := 
r \sin(\theta)$. Hence,  the solutions to the differential equation in  \eqref{eigenprob} are of the form 
\begin{align*}
 z(x) := & ~ C_1   \exp^{\omega_4 x} \sin (\omega_5 x) + C_2 \exp^{\omega_4 x} \cos (\omega_5 x) 
\\ & 
 + C_3 \exp^{- \omega_4 x} \sin (\omega_5 x) + C_4 \exp^{-\omega_4 x} \cos (\omega_5 x),   
 \end{align*}
where $(C_1,C_2, C_3,C_4)$ are free parameters subject to the boundary conditions in \eqref{eqBC}.  Hence,  the solution $z$ is nontrivial if and only if 
$$ \mathcal{A}_2 \begin{bmatrix}  C_1 & C_2 & C_3 & C_4  \end{bmatrix}^\top  = 0,  $$  
where 
\begin{align} \label{eqA2mat}
\mathcal{A}_2 := 
\begin{bmatrix}       
\mathcal{A}_{21} & \mathcal{A}_{22} & \mathcal{A}_{23} & \mathcal{A}_{24}
\end{bmatrix},
\end{align}
\begin{align*}
\mathcal{A}_{21} := 
\begin{bmatrix}       
0 
\\
  \exp^{\omega_4} \sin (\omega_5) 
\\
\omega_5 
 \\ 
 \omega_4  \exp^{\omega_4} \sin (\omega_5) + \omega_5 \exp^{\omega_4} \cos (\omega_5)
\end{bmatrix}
\end{align*}
\begin{align*}
\mathcal{A}_{22} := 
\begin{bmatrix}       
 1 
\\
 \exp^{\omega_4} \cos (\omega_5) 
\\
 \omega_4 
 \\ 
\omega_4  \exp^{\omega_4} \cos (\omega_5)  - 
\omega_5 \exp^{\omega_4} \sin (\omega_5)
\end{bmatrix}
\end{align*}
\begin{align*}
\mathcal{A}_{23} := 
\begin{bmatrix}       
0 \\
\exp^{- \omega_4} \sin (\omega_5) \\
\omega_5   \\ 
- \omega_4 \exp^{- \omega_4} \sin (\omega_5) +  \omega_5 \exp^{- \omega_4} \cos (\omega_5)
\end{bmatrix}
\end{align*}
\begin{align*}
\mathcal{A}_{24} := 
\begin{bmatrix}       
1  \\  
\exp^{-\omega_4} \cos (\omega_5) \\
-\omega_4 \\ 
- \omega_4  \exp^{-\omega_4} \cos (\omega_5) - \omega_5 \exp^{-\omega_4} \sin (\omega_5)
\end{bmatrix}.
\end{align*}
\end{enumerate}
As a consequence,  solving Problem \ref{prob1} is equivalent to solving the following optimization problem:
\begin{equation}
\label{eqoptim}
\begin{aligned}
\delta_o &  := \min \{ \delta_{o1},\delta_{o2} \},
\\
 \delta_{o1} & := \min \{ \delta \in [- \lambda^2/4,0] : \det \mathcal{A}_1 = 0  \},  
\\ 
 \delta_{o2} & := \min \{ \delta \leq - \lambda^2/4 : \det \mathcal{A}_2 = 0  \}. 
\end{aligned}
\end{equation}

Note that it  is easy to find $\delta_{o1}$ by plotting the curve 
of $\det \mathcal{A}_1$ as function of $\delta$ on the compact interval 
$[- \lambda^2/4,0]$.  However,  $\delta_{o2}$ is more difficult to obtain as it is the smallest root of a nonlinear function on an unbounded domain.  
To handle this issue,  we propose a method that gives a rough lower bound of $\delta_{o2}$,  denoted $\delta_{o2}^\star \leq - \lambda^2 /4$.   
As a result,  solving Problem \ref{prob1} becomes equivalent to solving the following optimization problem:
\begin{equation}
\label{eqoptim1}
\begin{aligned}
\delta_o &  := \min \{ \delta_{o1},\delta_{o2} \},
\\
 \delta_{o1} & := \min \{ \delta \in [- \lambda^2/4,0] : \det \mathcal{A}_1 = 0  \},  
\\ 
 \delta_{o2} & := \min \{ \delta \leq [ \delta^\star_{o2} , - \lambda^2/4] : \det \mathcal{A}_2 = 0  \}. 
\end{aligned}
\end{equation}

To compute $\delta^{\star}_{o2}$,  we first assume, without loss of generality,  that $\omega_4 > 0$. 
Next,  we decompose the matrix $\mathcal{A}_2$ as follows:
$\mathcal{A}_2  := \mathcal{A}_{2o} + \tilde{\mathcal{A}}_2$, where 
$\mathcal{A}_{2o} := 
\begin{bmatrix}       
\mathcal{A}_{21o} & \mathcal{A}_{22o} & \mathcal{A}_{23o} & \mathcal{A}_{24o}
\end{bmatrix}$,
\begin{align*}
\mathcal{A}_{21o} := 
\begin{bmatrix}       
0 
\\
  \exp^{\omega_5} \sin (\omega_5) 
\\
\omega_5 
 \\ 
 \omega_5  \exp^{\omega_5}  
 \left( \sin (\omega_5) +  \cos (\omega_5) \right)
\end{bmatrix},
\end{align*}
\begin{align*}
\mathcal{A}_{22o} := 
\begin{bmatrix}       
 1 
\\
 \exp^{\omega_5} \cos (\omega_5) 
\\
 \omega_5
 \\ 
\omega_5  \exp^{\omega_5} 
\left( \cos (\omega_5)  -  \sin (\omega_5) \right)
\end{bmatrix},
 ~~
\mathcal{A}_{23o} := 
\begin{bmatrix}       
 0
\\
0
\\
  \omega_5  
 \\ 
 0
\end{bmatrix}
, \quad 
\mathcal{A}_{24o} := 
\begin{bmatrix}       
1 
\\  
0
\\
  -\omega_5
 \\ 
0
\end{bmatrix},
\end{align*}
and $\tilde{\mathcal{A}}_{2} := 
\begin{bmatrix}       
\tilde{\mathcal{A}}_{21} & \tilde{\mathcal{A}}_{22} & 
\tilde{\mathcal{A}}_{23} & \tilde{\mathcal{A}}_{24}
\end{bmatrix}$,

\begin{align*}
\tilde{\mathcal{A}}_{21} := 
\begin{bmatrix}       
0 
\\
\left( \exp^{\omega_4} - \exp^{\omega_5} \right)  \sin (\omega_5) 
\\
0
 \\ 
\begin{bmatrix} 
\left( \omega_4  \exp^{\omega_4}  -   \omega_5  \exp^{\omega_5} \right) \sin (\omega_5) 
\\
+ \omega_5  \left( \exp^{\omega_4} - \exp^{\omega_5} \right)  \cos (\omega_5)
\end{bmatrix}
\end{bmatrix}
\end{align*}
\begin{align*}
\tilde{\mathcal{A}}_{22} := 
\begin{bmatrix}       
 0
\\
\left( \exp^{\omega_4}  -   \exp^{\omega_5} \right)  \cos (\omega_5) 
\\
 \omega_4 -  \omega_5
 \\ 
 \begin{bmatrix}
\left( \omega_4  \exp^{\omega_4} - \omega_5  \exp^{\omega_5}  \right)   \cos (\omega_5)  
\\
-  \omega_5 \left( \exp^{\omega_4}  - \exp^{\omega_5} \right)  \sin (\omega_5)
\end{bmatrix}
\end{bmatrix}
\end{align*}
\begin{align*}
\tilde{\mathcal{A}}_{23} := 
\begin{bmatrix}       
 0
\\
  \exp^{- \omega_4} \sin (\omega_5) 
\\
0
 \\ 
 - \exp^{- \omega_4} \left( \omega_4  \sin (\omega_5) -  \omega_5 
 \cos (\omega_5) \right)
\end{bmatrix}
\end{align*}
\begin{align*}
\tilde{\mathcal{A}}_{24} := 
\begin{bmatrix}       
0
\\  
 \exp^{-\omega_4} \cos (\omega_5)
\\
 \omega_5 - \omega_4
 \\ 
- \exp^{-\omega_4} \left( \omega_4   \cos (\omega_5) + \omega_5  
\sin (\omega_5) \right)
\end{bmatrix}.
\end{align*}
Now, we note that 
$\det \mathcal{A}_{2o} = -\omega_5^2 \exp^{2 \omega_5} < 0$.
Next, we introduce the following
notation:
\begin{align*}
a &  :=  \exp^{\omega_4} - \exp^{\omega_5},  ~
b :=  \omega_4  \exp^{\omega_4}  -   \omega_5  \exp^{\omega_5}, 
\\  
c & :=  \omega_4 -  \omega_5,
\quad d  := \exp^{-\omega_4}, ~ e  :=  b \sin (\omega_5) + a \omega_5 \cos(\omega_5) 
\\
f & := b \cos (\omega_5) - a \omega_5 \sin(\omega_5) 
\\
g & := - d  c \sin(\omega_5)   - d  \omega_5  \left[ \sin(\omega_5) - \cos(\omega_5)    \right]  
\\
h & := - d  c \cos(\omega_5)   - d  \omega_5  \left[ \cos(\omega_5) + \sin(\omega_5)    \right].
\end{align*}
Furthermore, we note that
$ |c|  = \frac{\lambda}{ \left(  \sqrt{-4\delta}  + \lambda \right)^{\frac{1}{2}}  + \left(  \sqrt{-4\delta}  - \lambda  \right)^{\frac{1}{2}} }$ and
 $ \omega_5 = \frac{1}{2}  \left[  \frac{- \lambda^2 - 4 \delta}{\sqrt{-4 \delta} + \lambda}   \right]^{\frac{1}{2}}$. Hence, we know how to make $|c|$ arbitrarily small by choosing $-\delta$ sufficiently large.  Next, we consider the following inequalities:  
\begin{equation}
\label{eqinequ}
\begin{aligned}
\hspace{-0.2cm} |a| \leq & ~  \exp^{\omega_5}  \left(  \exp^{|c|} - 1  \right),
\\
\hspace{-0.2cm} |b|  \leq & ~   |c|  \exp^{|c|} \exp^{\omega_5}  +\left( \exp^{|c|} - 1  \right) \omega_5 \exp^{\omega_5},
\\
\hspace{-0.2cm} |d| \leq & ~ \exp^{|c|} \exp^{-\omega_5},
\\
\hspace{-0.2cm} \max\{ |e|, |f|\} \leq & ~ |c|  \exp^{|c|} \exp^{\omega_5}  +2 \left( \exp^{|c|} - 1  \right) \omega_5 \exp^{\omega_5},
\\
\hspace{-0.2cm} \max\{ |g|, |h| \} \leq & ~ \exp^{-\omega_5}    \left(\exp^{|c|} |c| + 2 \exp^{|c|}  \omega_5 \right).
\end{aligned}
\end{equation}
On the other hand,  the matrix $\tilde{\mathcal{A}}_2$ can be expressed as  
\begin{align*}
\tilde{\mathcal{A}}_2 := 
\begin{bmatrix}
0 &  0 & 0 & 0
\\
a \sin (\omega_5) & a \cos(\omega_5) &  d \sin(\omega_5) & d \cos(\omega_5)
\\
0 & c & 0 &  -c
\\
e &  f & g & h 
\end{bmatrix}.
\end{align*}

Using the new notation,   $\det \mathcal{A}_2$ is given by
\begin{align*}
&  \det \mathcal{A}_2  = - 2 \exp^{2 \omega_5} \omega_5^2 + \exp^{\omega_5} \omega_5  \sin (\omega_5)   f
 \\ & + a \omega_5^2 \sin(\omega_5)  \exp^{\omega_5} (\cos(\omega_5) - \sin(\omega_5)) 
 \\ & + a \omega_5 \sin(\omega_5)  f   - \omega_5^2 \exp^{\omega_5} (\sin(\omega_5) + \cos(\omega_5))  a \cos(\omega_5) 
\\ & -   e \omega_5 \exp^{\omega_5} \cos(\omega_5)  
-  e \omega_5  a \cos(\omega_5) 
\\ & - d \sin (\omega_5)  \omega_5 ( \omega_5 \exp^{\omega_5} 
  ( \cos(\omega_5) - \sin (\omega_5)) + f)      
\\ & + d \sin (\omega_5) (\omega_5 + c)  (\omega_5  \exp^{\omega_5} (\sin(\omega_5)  + \cos(\omega_5)) + e )
\\ & + g (\exp^{\omega_5}  + a) \sin(\omega_5) (\omega_5 + c)  - g \omega_5 (\exp^{\omega_5} + a) \cos(\omega_5)
\\ & - (\exp^{\omega_5} + a) \sin(\omega_5) (\omega_5 h + g(\omega_5 + c)) +  d \sin(\omega_5) \omega_5 h 
\\ & + d \sin(\omega_5)  (\omega_5 + c) 
(\omega_5 \exp^{\omega_5} (\sin(\omega_5) + \cos(\omega_5)) +   e )  
\\ &  - d \cos(\omega_5) \omega_5 g + d \cos(\omega_5) \omega_5 e + d \cos(\omega_5) \omega_5^2 \exp^{\omega_5} (\sin(\omega_5) + \cos(\omega_5)).
\end{align*}
Next,  since $\omega_5  > 0$, we obtain
\begin{align*}
&  \det \mathcal{A}_2  \leq - 2 \exp^{2 \omega_5} \omega_5^2 + \exp^{\omega_5} \omega_5 | f|
 \\ & + 2 |a| \omega_5^2   \exp^{\omega_5}  + |a| \omega_5  |f|   + \omega_5^2 \exp^{\omega_5} 2  |a| 
 +  |e| \omega_5 \exp^{\omega_5}   +  |e| \omega_5  |a|  + d   \omega_5 ( 2 \omega_5 \exp^{\omega_5} + |f|)      
\\ & + d (\omega_5 + |c|)  (2 \omega_5  \exp^{\omega_5} + |e| )
 + |g| (\exp^{\omega_5}  + |a|) \ (\omega_5 + |c|)  +  |g|  \omega_5 (\exp^{\omega_5} + |a|) 
\\ & + (\exp^{\omega_5} + |a|)  (\omega_5 |h| + |g| \omega_5 + |c| |g| )
\\ & +  d \omega_5 |h| 
 + d   (\omega_5 + |c|) 
(2 \omega_5 \exp^{\omega_5}  +   |e| )  
  + d  \omega_5 |g| + d  \omega_5 |e| + 2 d  \omega_5^2 \exp^{\omega_5}.
\end{align*}
Now, using \eqref{eqinequ}, we conclude that
\begin{align*}
& \det \mathcal{A}_2  \leq - 2 \exp^{2 \omega_5} \omega_5^2 +12  \left( \exp^{|c|} - 1  \right) \exp^{2\omega_5} \omega_5^2   
 \\ & + 2 \left( |c|  \exp^{|c|}  + \left(  \exp^{|c|} - 1  \right)  |c|  \exp^{|c|}    \right)  \exp^{2\omega_5} \omega_5     
\\ & + \exp^{|c|}     \left( 2 \omega_5^2 + |c|  \exp^{|c|} \omega_5 +2 \left( \exp^{|c|} - 1  \right) \omega_5^2    \right)      
\\ & + \exp^{|c|}  \left( \omega_5 + |c|)  (2 \omega_5  + |c|  \exp^{|c|}  +2 \left( \exp^{|c|} - 1  \right) \omega_5    \right)
\\ & +   \left(\exp^{|c|} |c| + 2 \exp^{|c|}  \omega_5 \right) \left( 1 + \left(  \exp^{|c|} - 1  \right)  \right)  (\omega_5 + |c|) 
\\ & +   \left(\exp^{|c|} |c| + 2 \exp^{|c|}  \omega_5 \right)  \omega_5 \left( 1 +   \left(  \exp^{|c|} - 1  \right) \right) 
\\ & + \left(1 +   \left(  \exp^{|c|} - 1  \right) \right) 
 ( 2 \omega_5 +  |c| )   \left(\exp^{|c|} |c| + 2 \exp^{|c|}  \omega_5 \right)  
\\ & +  \exp^{|c|}  \omega_5    \left(\exp^{|c|} |c| + 2 \exp^{|c|}  \omega_5 \right)
 \\ & + \exp^{|c|}   (\omega_5 + |c|) 
\left( 2 \omega_5 + |c|  \exp^{|c|}   +2 \left( \exp^{|c|} - 1  \right) \omega_5   \right)  
\\ &  + \exp^{|c|} \omega_5   \left(\exp^{|c|} |c| + 2 \exp^{|c|}  \omega_5 \right)
\\ & +  |c|  \exp^{2|c|}   \omega_5  +2 \left( \exp^{|c|} - 1  \right) \exp^{|c|}  \omega_5^2   
 + 2 \exp^{|c|}   \omega_5^2.
\end{align*}
At this point, we choose $-\delta$ sufficiently large such that 
$\exp^{|c|} - 1 \leq 1$. Hence, we obtain 
\begin{align*}
& \det \mathcal{A}_2  \leq - 2 \exp^{2 \omega_5} \omega_5^2 +12 \left( \exp^{|c|} - 1 \right)  \exp^{2\omega_5} \omega_5^2   
\\ & + 12 \ln(2)    \exp^{2\omega_5} \omega_5   +    80 \ln(2)   \omega_5 + 56  \omega_5^2 + 12 \ln(2)^2.
\end{align*}
Next, we choose $-\delta$ even larger such that $12 \left( \exp^{|c|} - 1 \right) \leq 1$.
Hence, we obtain 
\begin{align*}
 \det \mathcal{A}_2  \leq  P(\omega_5) := & -  \exp^{2 \omega_5} \omega_5^2  + 12 \ln(2)    \exp^{2\omega_5} \omega_5  +   80 \ln(2)   \omega_5 + 56  \omega_5^2 + 12 \ln(2)^2.
\end{align*}
Clearly,  we can find  $\omega^{\star}_5>0$ such that,  for each 
$\omega_5 \geq \omega_5^\star$,  we have  
$P(\omega_5) \leq 0$. As a result,   a rough lower bound of $\delta_{o2}$, denoted $\delta^\star_{o2}$, can be chosen such that $ \sqrt{-4\delta^\star_{o2}}    = \left( \frac{\lambda} {2 \ln \left( 13/12 \right)}  \right)^2 + \lambda + 4 {\omega^\star_5}^2$. 
\end{proof}

\subsection{Proof of Lemma \ref{lemintegpart}}
Note that along the solutions to $\Sigma_2$, we have 
\begin{equation*} 
\begin{aligned} 
\dot{V}_1 = &  \int^{Y}_{0} w(x) w_t(x) dx =  -  \int^{Y}_{0} w(x)^2 w_x(x) dx  - \lambda_1  \int^{Y}_{0} w(x) w_{xx}(x) dx 
\\ & -  \int^{Y}_{0} w(x) w_{xxxx}(x) dx,
\\
\dot{V}_2 =  &  \int^{L}_{Y} v(x) v_t(x) dx =  -\int^{L}_{Y}  v(x)^2 v_x(x) dx  - \lambda_1 \int^{L}_{Y}  v(x) v_{xx}(x) dx \\ & - \int^{L}_{Y}  v(x) v_{xxxx}(x) dx.
\end{aligned}    
\end{equation*}
Next, we use the following identities
\begin{align*}
 \int^{Y}_{0} \hspace{-0.2cm} w(x)^2 w_x(x) dx &  = \left(w(Y)^3 - w(0)^3 \right)/3, 
\\
 \int^{Y}_{0} \hspace{-0.2cm} w(x) w_{xx}(x) dx & =   w(Y) w_x(Y) - 
w(0) w_x(0)   - \int^{Y}_{0} \hspace{-0.2cm} w_x(x)^2 dx, 
\\
\int^{Y}_{0} \hspace{-0.2cm} w(x) w_{xxxx}(x) dx  & = 
\int^{Y}_{0} \hspace{-0.2cm} w_{xx}(x)^2 dx   + w(Y) w_{xxx}(Y) 
\\ &  \qquad \qquad  
- w(0) w_{xxx}(0) - w_x(Y) w_{xx}(Y) +  w_x(0) w_{xx}(0). 
 \end{align*}
The latter three identities hold also if we replace $v$ therein by $w$. 
As a result,  we obtain
\begin{align*}
\dot{V}_1 = & -\frac{w(Y)^3 - w(0)^3}{3}  - 
\int^{Y}_{0} w_{xx}(x)^2 dx   
\\ & - \lambda_1 w(Y) w_x(Y) + \lambda_1 w(0) w_x(0)
 + \lambda_1 \int^{Y}_{0} w_x(x)^2 dx - 
w(Y) w_{xxx}(Y)  \\ & + w(0) w_{xxx}(0) + 
w_x(Y) w_{xx}(Y)   - w_x(0) w_{xx}(0), 
\\
\dot{V}_2 = & -\frac{v(L)^3 - v(Y)^3}{3}  - \int^{L}_{Y} v_{xx}(x)^2 dx   
\\ & 
- \lambda_1 v(L) v_x(L) + \lambda_1 v(Y) v_x(Y)
 + \lambda_1 \int^{L}_{Y} v_x(x)^2 dx - v(L) v_{xxx}(L) 
 \\ & +v(Y) v_{xxx}(Y) + v_x(L) v_{xx}(L)  - v_x(Y) v_{xx}(Y).  
\end{align*}
Finally,  using \eqref{eqDerich},  \eqref{eqLyapgen} follows. 
\hfill $\blacksquare$

\subsection{Proof of Lemma \ref{lem4}}
Given a scalar function $z : [a,b] \rightarrow \mathbb{R}$ in $H^2(a,b)$, we introduce the function $u : [a,b] \rightarrow \mathbb{R}$ given by
$$ u(x) := z(x) - \kappa(x) \qquad \forall x \in [a,b], $$
where $w$ is introduced in \eqref{eqomega}. Note that $u \in H^2_{or}(a,b)$ since $u(a) = z(a) - \kappa(a)$ and $\kappa(a) = z(a)$. Furthermore, $u(b) = z(b) - \kappa(b)$ and $\kappa(b) = z(b)$.  

As a result, using Lemma \ref{lem4-}, we conclude that, for $\lambda := 3\lambda_1$, we have 
\begin{align*}
- \int^{b}_{a} u_{xx}(x)^2 dx   + \lambda  \int^{b}_{a} u_x(x)^2 dx 
 \leq  - \delta \int^{b}_{a} u(x)^2 dx \qquad \forall \delta \leq \delta_o,
\end{align*}
where $\delta_o$ is the corresponding solution to Problem \ref{prob1}. The latter inequality implies  that 
\begin{align*}
 - \int^{b}_{a} [z_{xx}(x) - \kappa_{xx}(x)]^2 dx +  3\lambda_1  \int^{b}_{a} [z_{x}(x) - \kappa_{x}(x)]^2 dx  \leq 
-\delta \int^{b}_{a} [z(x) - \kappa(x)]^2 dx.
\end{align*}
Now, we use Lemma \ref{lem5} to derive the following inequalities 
$$ -[z_{xx}(x) - \kappa_{xx}(x)]^2 \geq 
- \frac{3}{2} z_{xx}(x)^2 - 3\kappa_{xx}(x)^2.  $$
$$ [z_{x}(x) - \kappa_{x}(x)]^2 \geq \frac{1}{2} z_{x}(x)^2 - \kappa_x(x)^2. $$
\begin{align*}
-\delta [z(x) - \kappa(x)]^2  \leq  \left(\frac{|\delta|}{2} - \delta \right) z(x)^2  + ( 2 |\delta| - \delta) \kappa(x)^2.
\end{align*}
As a consequence, we obtain
\begin{align*}
 - \frac{3}{2} \int^{b}_{a} z_{xx}(x)^2 dx   &
+ \frac{3\lambda_1}{2} \int^{b}_{a} z_{x}(x)^2 \\ &  \leq \left(\frac{|\delta|}{2} - \delta \right) \int^{b}_{a} z(x)^2 dx 
+ ( 2 |\delta| - \delta)  \int^{b}_{a} \kappa(x)^2 dx \\ & 
+ 3 \int^{b}_{a} \kappa_{xx}(x)^2 dx  + 3 \lambda_1  \int^{b}_{a} \kappa_{x}(x)^2 dx,
\end{align*}
which yields to
\begin{align*}
 - \frac{3}{2} \int^{b}_{a} z_{xx}(x)^2 dx  & + \frac{3\lambda_1}{2} \int^{b}_{a} z_{x}(x)^2 \\ &  
 \leq  \frac{3\delta_1}{2} \int^{b}_{a} z(x)^2 dx   
 + \frac{3\delta_2}{2}  \int^{b}_{a} \kappa(x)^2 dx + 3 \int^{b}_{a} \kappa_{xx}(x)^2 dx  \\ & + 3\lambda_1  \int^{b}_{a} \kappa_{x}(x)^2 dx. 
\end{align*}
Hence, we deduce that
\begin{align*}
 - \int^{b}_{a} z_{xx}(x)^2 dx & + \lambda_1 \int^{b}_{a} z_{x}(x)^2 
\\ &
\leq  
\delta_1 \int^{b}_{a} z(x)^2 dx +  
\delta_2 \int^{b}_{a} \kappa(x)^2 dx \\ & + 2 \int^{b}_{a} \kappa_{xx}(x)^2 dx + 2\lambda_1  \int^{b}_{a} \kappa_{x}(x)^2 dx. 
\end{align*}
Finally, using \eqref{eqCs}, we obtain 
\begin{align*}
 - \int^{b}_{a} z_{xx}(x)^2 dx +~& \lambda_1 \int^{b}_{a} z_{x}(x)^2 \leq  
\delta_1 \int^{b}_{a} z(x)^2 dx + \\ & 
\delta_2 C_{z2}(z(a),z(b)) +  C_{z1}(z(a),z(b)). 
\end{align*}
\hfill $\blacksquare$

\subsection{Proof of Lemma \ref{lemfd}}
A direct application of Lemma \ref{lem4} allows us to re-express \eqref{eqLyapgen} as 
\begin{equation*}
\begin{aligned} 
\hspace{-0.3cm} & - \int^{Y}_{0}  w_{xx}(x)^2 dx   + \lambda_1 \int^{Y}_{0} w_x(x)^2 dx  \leq   \delta_1 \int^{Y}_{0} w(x)^2 dx 
\hspace{-0.3cm} \\ &   \qquad \quad  +  C_{w1} (u_1,u_2) + \delta_2  C_{w2}(u_1,u_2)  + \lambda_1  C_{w3}(u_1,u_2).
\end{aligned}
\end{equation*}
Similarly, we have that
\begin{equation} 
\label{eqineqv} 
\begin{aligned} 
\hspace{-0.2cm} - \int^{L}_{Y}  & v_{xx}(x)^2 dx   + \lambda_1 \int^{L}_{Y} v_x(x)^2 dx  \leq  \delta_1 \int^{L}_{Y} v(x)^2 dx 
\\ &
\hspace{-0.2cm}  + C_{v1} (u_2,u_3) + \delta_2  C_{v2}(u_2,u_3) + \lambda_1 C_{v3}(u_2,u_3). 
\end{aligned}
\end{equation}
 As a result, we obtain
\begin{equation*}
\begin{aligned}
\dot{V}_1 \leq &~  \delta_1 \int^{Y}_{0} w(x)^2 dx  + C_{w1} (u_1,u_2) + \delta_2  C_{w2}(u_1,u_2)
\\ &
 + \lambda_1  C_{w3}(u_1,u_2) -  \frac{u_2^3 - u_1^3}{3} - u_2 w_{xxx}(Y) + u_1 w_{xxx}(0), 
\\
\dot{V}_2 \leq  &~   \delta_1 \int^{Y}_{0} v(x)^2 dx   + C_{v1} (u_2,u_3) + \delta_2  C_{v2}(u_2,u_3) 
\\ &
+ \lambda_1  C_{v3}(u_2,u_3) -  \frac{u_3^3 - u_2^3}{3} - u_3 v_{xxx}(L)  +
u_2 v_{xxx}(Y).  
\end{aligned}
\end{equation*}
Finally,  using \eqref{eq.4}, we obtain \eqref{eqLyapgenbis1}.
\hfill $\blacksquare$

\subsection{Proof of Lemma \ref{lemdesign}}

We distinguish between two complementary situations.
 When 
\begin{align*}
|w_{xxx}(0)| & \geq l_1(V_1) :=  \frac{V_1^2}{3} + (a_{w1}  + \delta_2 a_{w2})  V_1  + (\alpha_1 + 2 \delta_1);
\end{align*}
namely, when
$$ V_1 |w_{xxx}(0)|  \geq \frac{V_1^3}{3} + (a_{w1}  + \delta_2 a_{w2} + \lambda_1 a_{w3})  V_1^2  + (\alpha_1 + 2 \delta_1) V_1, $$
we take  $ \kappa_1(V_1,w_{xxx}(0)) := - \text{sign} (w_{xxx}(0)) V_1 $ so that \eqref{eqdesign1} is satisfied. 
 Otherwise, when $|w_{xxx}(0)|  < l_1(V_1)$,  we choose $\kappa_1$ such that
\begin{equation}
\label{eqineq1}
\begin{aligned} 
\kappa^3_1 & +3 (a_{w1}  + \delta_2 a_{w2} + \lambda_1 a_{w3})  \kappa_1^2 + 3 |\kappa_1| |l_1(V_1)|  
 \leq  -3(\alpha_1 + 2 \delta_1) V_1,
\end{aligned}
\end{equation}
which implies that \eqref{eqdesign1} is satisfied. 

The same reasoning applies when designing $u_3$  to satisfy \eqref{eqdesign2}.  We distinguish between two complementary situations. When 
\begin{align*}
|v_{xxx}(L)| & \geq l_3(V_2) :=  \frac{V_2^2}{3} + (b_{v1}  + \delta_2 b_{v2} + \lambda_1 b_{v3} )  V_2  + (\alpha_2 + 2 \delta_1);
\end{align*}
namely, when
$$ V_2 |v_{xxx}(L)|  \geq \frac{V_2^3}{3} + (b_{v1}  + \delta_2 b_{v2}  + \lambda_1 b_{v3} )  V_2^2  + (\alpha_2 + 2 \delta_1) V_2, $$
we take $ \kappa_3(V_2,v_{xxx}(L)) := - \text{sign} (v_{xxx}(L)) V_2 $
so that \eqref{eqdesign2} is satisfied. 
Otherwise, when $|v_{xxx}(L)|  < l_3(V_2)$,  we choose $\kappa_3$ such that
\begin{equation}
\label{eqineq1bis}
\begin{aligned} 
\kappa^3_3  +3 (b_{v1}  + \delta_2 b_{v2} + \lambda_1 b_{v3})  \kappa_3^2 & + 3 |\kappa_3| |l_3(V_2)|   \leq  -3(\alpha_2 + 2 \delta_1) V_2,
\end{aligned}
\end{equation}
 which implies that \eqref{eqdesign2} is satisfied. 
\hfill $\blacksquare$

\subsection{Proof of Lemma \ref{lemwxxx}}
We re-express the term $w_{xxx}(0)$ as
\begin{align*}
 w_{xxx}(0) = w_{xxx}(Y) - \int^{Y}_{0} w_{xxxx} (x) dx 
 \end{align*}
Now, using  the first equation in $\Sigma_2$, we conclude that
 \begin{align*}
 w_{xxx}(0) & =  w_{xxx}(Y)  +  \int^{Y}_{0} \left[w_t(x) + \lambda_1 w_{xx}(x) + w(x) w_x(x)\right] dx
\\ & =  w_{xxx}(Y) +  \int^{Y}_{0} w_t(x) dx + \lambda_1  \int^{Y}_{0} w_{xx}(x) dx + \int^{Y}_{0} w(x) w_x(x) dx
\\ & = w_{xxx}(Y) +  \int^{Y}_{0} w_t(x) dx + \lambda_1 [w_{x}(Y) - w_{x}(0)]  +  [w(Y)^2 - w(0)^2]/2. 
\end{align*} 
 Using the fact that $w_x(Y) = w_x(0) = 0$, $w(0)=u_1$, and $w(Y) = u_2$, we obtain 
 \begin{align*}
w_{xxx}(0) = & ~ w_{xxx}(Y) +  \int^{Y}_{0} w_t(x) dx + \left(u^2_2 - u^2_1\right)/2
\\ = & ~  w_{xxx}(Y) +  \left( u^2_2 - u^2_1 \right)/2  
+ \frac{d}{dt} \left(\int^{Y}_{0} w(x) dx \right). 
\end{align*}
\hfill $\blacksquare$

\subsection{Proof of Lemma \ref{lemdesu2}}
We start letting $\zeta := v_{xxx}(Y)$.  Next, we note that the input $u_2$ needs to satisfy the inequality
\begin{equation} \label{equ2}
\begin{aligned}
 u^3_2 +  3u_2 \zeta \leq -3 \alpha_2 V_2^3.
\end{aligned}
\end{equation}
 For this,  we distinguish between the following two situations.
\begin{enumerate}
\item When $|\zeta|   \geq  k (\alpha_2^{\frac{1}{3}} V_2)  := 2 \alpha_2^{\frac{2}{3}} V_2^2$,
we conclude that $\alpha_2^{\frac{1}{3}} V_2  |\zeta|  \geq  2 \alpha_2 V_2^3$.
Hence,  by taking
$u_2 := \kappa_2 \left(\alpha_2^{\frac{1}{3}} V_2, \zeta \right) : =  - \text{sign} (\zeta) \alpha_2^{\frac{1}{3}} V_2$, we conclude that \eqref{equ2} is satisfied. 

\item In the opposite scenario; namely, when $|\zeta|  <  k(\alpha_2^{\frac{1}{3}} V_2)$, 
we choose $u_2 := \kappa(\alpha_2^{\frac{1}{3}} V_2)$, where $\kappa$ is  solution to
\begin{align} \label{eqtosat}
\kappa(\alpha_2^{\frac{1}{3}} V_2)^3 + 3|\kappa(\alpha_2^{\frac{1}{3}} V_2)|  k(\alpha_2^{\frac{1}{3}} V_2)  \leq -3 \alpha_2 V_2^3. 
\end{align}
Indeed, the latter inequality implies that 
$$ u^3_2 + 3|u_2| |\zeta|  \leq -3 \alpha_2 V_2^3.  $$ 
Hence,  \eqref{equ2} is satisfied. Now, to solve \eqref{eqtosat} while verifying Assumption \ref{ass2}, we take $\kappa(\alpha_2^{\frac{1}{3}} V_2) := -\beta \alpha_2^{\frac{1}{3}} V_2$ and we choose $\beta$ sufficiently large such that  
\begin{align*}
- \beta^3 \alpha_2^{\frac{1}{3}} V_2^3 + 3 \beta \alpha_2^{\frac{1}{3}} V_2 k(\alpha_2^{\frac{1}{3}} V_2) \leq - 3 \alpha_2^{\frac{1}{3}} V_2^3 . 
\end{align*}
\end{enumerate}
As a result, we design the input $\kappa_2$ to satisfy 
\begin{equation}
\label{eqcontrol}
\begin{aligned}
\kappa_2(\alpha_2^{\frac{1}{3}} V_2,\zeta) := 
\left\{ \begin{matrix}
  - \text{sign} (\zeta) \alpha_2^{\frac{1}{3}} V_2 & \text{if} ~ |\zeta| \geq 2 \alpha_2^{\frac{1}{3}} V_2^2 \\
- \beta \alpha_2^{\frac{1}{3}} V_2 & \text{otherwise}.
\end{matrix} \right.
\end{aligned}
\end{equation}
To verify Assumption \ref{ass2}, we note that \eqref{asscond1} is satisfied with $P := \beta + 1$. Next, to verify \eqref{asscond2}, we consider an absolutely continuous signal  
 $t \rightarrow \left(\alpha_2^{\frac{1}{3}} V_2(t), \zeta(t) \right)$. According to \eqref{eqcontrol}, $ t \mapsto u_2(t) = \kappa \left( \alpha_2^{\frac{1}{3}} V_2(t),\zeta(t) \right)$ is differentiable almost everywhere and, for almost all $t$,  we have 
\begin{align*} 
 \bigg| \frac{d}{dt} u_2(t) \bigg| & \leq \bigg| \frac{\partial u_2}{\partial  \left(\alpha_2^{\frac{1}{3}} V_2\right) } (t) \bigg|   \bigg| \frac{d}{dt} \alpha_2^{\frac{1}{3}} V_2(t) \bigg|,
 \end{align*} 
where, for almost all $t$,  $\bigg| \frac{\partial u_2}{\partial \left(\alpha_2^{\frac{1}{3}} V_2\right) } (t) \bigg|  \leq  1 + \beta =: Q$.
\hfill $\blacksquare$

\subsection{Proof of Lemma \ref{lem8}}
Consider the function 
$$ f(\pi) := \pi^3  - \left( 2 \delta_1  \alpha_2^{-\frac{1}{3}}+ 1 \right) \pi.  $$
Note that the derivative of $f$ with respect to $\pi$ is given by 
$$ f'(\pi) = 3 \pi^2  - \left( 2 \delta_1 \alpha_2^{-\frac{1}{3}} + 1 \right).   $$
Hence,  the equation $f'(\pi) = 0$ has a unique solution $\pi^\star$ given by 
$$ \pi^\star  := \frac{ \sqrt{3 \left( 2 \delta_1 \alpha_2^{-\frac{1}{3}} + 1 \right) }}{3} $$  
at which $f$ attains its minimum value.  
Hence,  for \eqref{eqlem81} to hold,  we take 
\begin{equation*} 
\begin{aligned}
& b_o :=  f(\pi^\star) =   - \left(\frac{ \left[3 (2 \delta_1 \alpha_2^{-\frac{1}{3}} + 1)  \right]^{\frac{1}{2}}}{3}  \right)^3  +
\frac{ 2 \delta_1 \alpha_2^{-\frac{1}{3}}  +1 }{3}  \left( \left[ 3 (2 \delta_1 \alpha_2^{-\frac{1}{3}} + 1)  \right]^{\frac{1}{2}} \right). 
\end{aligned}
\end{equation*}
\hfill $\blacksquare$
\fi

\end{document}